\documentclass[11pt]{article}
\usepackage{amsmath, amsfonts, amsthm, amssymb}
\usepackage[]{pstricks,pst-node,pst-text,pst-3d}

\usepackage{geometry}
\usepackage{graphicx,float,subcaption}
\usepackage{authblk}
\usepackage{pgfkeys}
\usepackage{fullpage,tocbibind}

\usepackage{hyperref}

\usepackage{ulem}  

\usepackage{courier}
\newcommand{\JJ}{\textbf{J}}
\newcommand{\K}{\textbf{K}}
\newcommand{\N}{\textbf{N}}
\newcommand{\R}{\mathbb{R}}
\newcommand{\SSS}{\mathcal{S}}

\newcommand{\EE}{\mathbb{E}}
\newcommand{\NN}{\mathbb{N}}
\newcommand{\phat}{\hat{p}}
\newcommand{\yhat}{\hat{y}}
\newcommand{\ybar}{\bar{y}}
\newcommand{\pbar}{\bar{p}}
\newcommand{\ptil}{\tilde{p}}
\newcommand{\lamhat}{\hat{\lambda}}
\newcommand{\Omegaline}{\underline{\Omega}}

\newcommand{\ds}{\displaystyle}
\newcommand{\PP}{\mathbb{P}}
\newcommand{\MISE}{\mbox{\rm MISE}}

\newcommand{\EXP}{\mathbb{E}}

\definecolor{mblue}{rgb}{0,0,1}
\definecolor{cmtcol}{rgb}{0.5,0,1}
\definecolor{dbtcol}{rgb}{0.75,0.0,0.0}

\title{Asymptotic properties and approximation of Bayesian logspline density estimators for communication-free parallel computing methods}

\author{ Konstandinos Kotsiopoulos\thanks{University of Massachusetts Amherst,  kkotsiop@gmail.com, ORCID:0000-0003-2651-0087}, Alexey Miroshnikov\thanks{University of Massachusetts Amherst, University of California Los Angeles, amiroshn@gmail.com, ORCID:0000-0003-2669-6336} 
and Erin Conlon\thanks{University of Massachusetts Amherst, econlon@umass.edu} }

\newcommand{\dist}{{\mbox{\rm dist}}}

\newcommand{\RR}{\mathbb{R}}
\newcommand{\eps}{\varepsilon}

\def\clap#1{\hbox to 0pt{\hss#1\hss}}

\newtheorem{theorem}{Theorem}[section]
\newtheorem{definition}{Definition}[section]
\newtheorem{lemma}{Lemma}[section]
\newtheorem{remark}{Remark}[section]
\newtheorem{proposition}{Proposition}[section]

\newcommand{\vertiii}[1]{{\left\vert\kern-0.25ex\left\vert\kern-0.25ex\left\vert #1
\right\vert\kern-0.25ex\right\vert\kern-0.25ex\right\vert}}

\date{}

\begin{document}
	
	\maketitle
	
\abstract
{
	In this article we perform an asymptotic analysis of parallel Bayesian logspline density estimators. Such estimators are useful for the analysis of datasets that are partitioned into subsets and stored in separate databases without the capability of accessing the full dataset from a single computer. The parallel estimator we introduce is in the spirit of a kernel density estimator introduced in recent studies. We provide a numerical procedure that produces the normalized density estimator itself in place of the sampling algorithm. We then derive an error bound for the mean integrated squared error of the full dataset posterior estimator. The error bound depends upon the parameters that arise in logspline density estimation and the numerical approximation procedure. In our analysis, we identify the choices for the parameters that result in the error bound scaling optimally in relation to the number of samples. This provides our method with increased estimation accuracy, while also minimizing the computational cost.
	}
	
	\vspace{11pt}
	
	{{\bf Keywords:} logspline density estimation, asymptotic properties, parallel algorithms.}
	
	\vspace{11pt}
	
	{{\bf Mathematics Subject Classification (2010): }62G07, 62G20, 68W10.}
	

\section{Introduction}

The recent advances in data science and big data research have brought challenges in analyzing large data sets in full. In particular, carrying out an analysis may be time-consuming given the size of the dataset. More importantly, datasets that constitute one large dataset may be scattered across different locations. Occurrences like this may happen in several industries. For example, medical institutions keep their data private and merging these datasets into one is often not possible. This motivates us to investigate algorithms that extract information from each individual subset and combine the results to gain insight into the larger dataset. 

A proposed framework for such an analysis is through recently developed parallel computing methods. One such approach assumes that a given dataset is partitioned into a number of subsets, where each subset is analyzed on a separate machine using parallel Markov chain Monte Carlo (MCMC) methods \cite{Langford,Newman,Smola}. In these methods, it is assumed that the subsets are stored in the same location and can be accessed by any processor. Furthermore, communication between machines is required for each MCMC iteration, which can be quite costly.

Due to the limitations of requiring datasets to be fully stored in one location and the limitations of methods that require communication between machines, a number of alternative communication-free parallel MCMC methods have been developed for Bayesian analysis of big data. For these approaches, Bayesian MCMC analysis is performed on each subset independently, and the subset posterior samples are combined to estimate the full data posterior distributions; see \cite{Neiswanger,Wang,Scott2016,Miroshnikov-Conlon}.

For the approaches introduced in the above literature, the full dataset posterior $p(\theta|\textbf{x})$ is reconstructed by estimating posterior densities $p_m(\theta) = p(\theta|\textbf{x}_{m})$ as follows
\begin{equation*}\label{fulldataintro}
	p(\theta|\textbf{x})\propto \ds\prod_{m=1}^{M}p_{m}(\theta)=: p^*(\theta)\,, \quad \text{where} \quad p_{m}(\theta):=p(\theta|\textbf{x}_{m})=p(\theta)^{1/M}p(\textbf{x}_m | \theta).
\end{equation*}
Here, $\theta=(\theta_1,\theta_2,\dots,\theta_d)$ are the model parameters, and $\textbf{x} = \{\textbf{x}_1, \textbf{x}_2,\dots, \textbf{x}_M\}$ is the full data set partitioned into $M$ disjoint independent subsets. Setting $c=\int p^*(\theta) \, d\theta$, we have $p(\theta|\textbf{x})=cp^*(\theta)$.

The algorithm of Neiswanger et al. \cite{Neiswanger} first approximates each subset posterior density in parallel using kernel density estimation \cite{Silverman} and then builds the estimator for the full data posterior by multiplying together the subset posterior estimators
\begin{equation}\label{fullestintro}
	\hat{p}(\theta|\textbf{x})\propto \hat{p}^*(\theta|\textbf{x}):= \hat{p}_{1}(\theta|\textbf{x}_1)\cdot \hat{p}_{2}(\theta|\textbf{x}_2) \dots \hat{p}_{M}(\theta|\textbf{x}_M),
\end{equation}
where $\hat{p}_m(\theta|\textbf{x}_m)$ is the kernel density estimator of $p_m(\theta|\textbf{x}_m)$.

Neiswanger et al. \cite{Neiswanger} show that the estimator \eqref{fullestintro} is asymptotically exact and the authors develop a sampling algorithm that generates samples from the distribution approximating the full dataset estimator. Similar sampling algorithms were presented and investigated in Wang and Dunson \cite{Wang}, Scott et al. (2016) \cite{Scott2016}, and Scott (2017) \cite{Scott2017}. It has been noted that these algorithms do not perform well for posteriors that have non-Gaussian shape (in practice, this occurs when either the prior is non-Gaussian and/or some of the data subsets are small) and are sensitive to the choice of the kernel parameters; see Miroshnikov et al. \cite{Miroshnikov-Conlon}. The works of Neiswanger et al. \cite{Neiswanger}, and Wang and Dunson \cite{Wang} provide strong theoretical guarantees as the authors provide the bounds for the estimation errors for the unnormalized densities, while omitting the investigation of the normalization constant.

Another important development is the article of Miroshnikov and Savelev \cite{Miroshnikov-Savelev}. This work focuses on the case $M>1$ and $d=1$ and carries out the asymptotic analysis of the parallel Bayesian kernel density estimator of the form \eqref{fullestintro} by performing an expansion of the mean integrated squared error (MISE) for both unnormalized and normalized density estimators $\phat^*$ and $\phat$, respectively, as the number of subset posterior samples $N = (N_1, N_2,\dots, N_M) \to \infty$. This expansion allows the authors to investigate the optimality of kernel bandwidth parameters, which aids in improving the estimation accuracy.  The analysis in \cite{Miroshnikov-Savelev} can be carried out to the case $d>0$, when the Bayesian parameters $\theta=(\theta_1,\dots,\theta_d)$ are independent. In this case,  the representation \eqref{fullestintro} for $\theta$ holds for each individual parameter $\theta_i$ and, as a consequence, one can construct the estimator for the joint posterior density using one-dimensional full posterior estimators.

The aforementioned papers focus on parallel Bayesian estimators in the form \eqref{fullestintro}, that are based on kernel density estimators, which are known to be sensitive to the choice of bandwidth parameters and, hence, less robust. For this reason, in this article we replace the kernel density estimators $\phat_m(\theta)$ with logspline estimators (where the information on $\textbf{x}$ is suppressed), which are known to perform well for asymmetric shapes; we then carry out a similar asymptotic analysis to that of Miroshnikov and Savelev \cite{Miroshnikov-Savelev}, based on the theory of logspline estimation of Stone \cite{Stone89,Stone90} and Stone and Koo \cite{StoneKoo} (which corresponds to $M=1$, $d=1$). These estimates depend on the $N_m$ number of samples drawn from the subset posterior density, the $K_m+1$ number of knots for the $k$-order B-splines, and the smallest knot distance $h_m$. As in \cite{Miroshnikov-Savelev}, we focus on the case of multiple subsets $M>1$ and $d=1$, where the analysis can be extended to $d>1$ under the assumption of independence of parameters $\{\theta_i\}_{i=1}^d$. Our work is partially motivated by Miroshnikov et al. \cite{Miroshnikov-Conlon}, where the authors focused on practical applications of logspline estimators in the form \eqref{fullestintro} to real-world datasets and illustrate that the method performs well for asymmetric posterior densities (reconstructing densities marginally for each parameter $\theta_i$). We also conduct numerical experiments to verify the theoretical results. The numerical comparison with other similar methods, however, is outside the scope of this paper and is the focus of future research.

The estimated posterior $\phat^{*}$ is unnormalized, which motivates one to define the normalization constant $\hat{c}=\int \phat^{*}(\theta)\, d\theta$. The normalized posterior estimator  $\phat$ then is given by $\phat(\theta)=\hat{c}^{-1}\phat^{*}(\theta)$ which is one of the points of interest in our work. Computing the normalization constant analytically is a difficult task since the subset posterior densities are not explicitly calculated.

The aim in our paper is to conduct an asymptotic analysis of the full density estimator $\phat(\theta)$, where the normalization constant needs to be addressed. In particular, our main objective is to establish an error estimate between the posterior density $p$ and its approximation $\phat$ via the mean integrated squared error (MISE) given by 
\begin{equation*}
{\MISE}(p,\phat):=\mathbb{E}\int \big(p(\theta)-\phat(\theta) \big) ^{2}d\theta.
\end{equation*}
Below is the summary of our main results:

\vspace{10 pt}

\begin{itemize}

		\item 

		 We estimate the MISE between the functions $\phat^{*}$ and $p^{*}$ by adapting the estimation techniques introduced in \cite{Stone89,Stone90}. Under optimal scaling of the number of knots with respect to the number of samples, we obtain the estimate
	\[
	\MISE(p^{*},\phat^{*})=O\left(M^{2-2\beta_{opt}}\|\N\|^{-2\beta_{opt}}\right), \quad \N=(N_1,\dots,N_M),
	\]
	where the parameter $\beta_{opt}=q/(2q+1)$; see Theorem  \ref{thm::bound_unnorm_est} below.

	\item 	To establish an estimate for the difference  between the normalized densities $\phat$ and $p$, we first carry out an analysis for the normalization constant $\hat{c}$, which gives
	\[
		\left| \frac{c}{\hat{c}(\omega)} - 1 \right|=O\left(M^{1-\beta_{opt}}\|\N\|^{-\beta_{opt}}\right).
	\]
This leads to the following desired estimate (see Theorem \ref{MISEpphat} below)
	\[
		\MISE(p,\phat)=O\left(M^{2-2\beta_{opt}}\|\N\|^{-2\beta_{opt}}\right).
	\]

\item Computing the constant $\hat{c}$ analytically is a difficult task since integrating the product of subset posterior density estimators explicitly is challenging. To circumvent this issue, we introduce an estimator $\ptil^*$ which is an interpolated version of $\phat^*$. This unnormalized estimator is straightforward to integrate numerically, which can provide us with the normalization constant $\tilde{c}$ of $\ptil^*$, which yields the full posterior estimator $\ptil=\tilde{c}^{-1}\ptil^*$.

We carry out a similar error analysis for $\phat$ and $\tilde{p}$, which leads to the final result for MISE between $p$ and $\ptil$ is given by
\begin{equation*}
	\MISE (p\,,\ptil)=O\left(M^{2-2\beta_{opt}}\|\N\|^{-2\beta_{opt}}\right),
\end{equation*}
where $l$ is the order of polynomials used to interpolate $\phat$, and the optimal interpolation step $\Delta x$ is chosen to scale optimally as $O\bigg( \|\N\|^{-\beta_{opt}\left(\frac{1}{l+1}+\frac{1}{q}\right)}\bigg)$; see Theorem \ref{MISEestthm} below.

\end{itemize}

The paper is arranged as follows. In Section \ref{section:intro material} we present the basics of logspline density estimation and notions of distance from splines. In Section \ref{section:parest} we introduce the parallel logspline estimators for the posterior density. In Section \ref{section:main_res} we state the main results that provide the error bounds for the unnormalized and normalized parallel estimators. In Section \ref{section::proofs} we provide the proofs of the main results via an asymptotic expansion of MISE. In Section \ref{section:num_exp} we showcase our method on both simulated and real-world datasets and discuss the results. Finally, in the Appendix we introduce notation and hypotheses that form the foundation of the analysis, as well as supplementary material.

\section{Preliminaries}\label{section:intro material}

	Here we provide all of the relevant results related to B-splines and logspline density estimators based on the works of \cite{de Boor,StoneKoo,Stone89,Stone90}.

	\subsection{Logspline family of estimators}\label{subsec:log_fam}
	In this part we will present the method for constructing logspline density estimators using B-splines. The reader is referred to the appendix for more information on B-splines. Let $p$ be a continuous probability density function supported on an interval $[a,b]$. Suppose $p$ is unknown and we would like to construct density estimators for this function. We start by introducing the reader to the concept of an exponential family of distributions \cite{Lehmann} since logspline density estimators are derived from such a family.
	
	\begin{definition}\label{exp_family}
		A family $\{P_y\}$ of distributions forms an $s$-dimensional exponential family if each distribution $P_y$ has a density $p_y$ of the form
		\begin{equation}\label{exp_fam_dens}
			p_y(x) = h(x) \exp\Big({\ds\sum_{j=0}^{s-1}\eta_{j}(y)T_{j}(x)-c(y)}\Big)
		\end{equation}
		where $y$ can be real or vector valued and $h,c$ and $\eta_j,T_j, \ j\in\{0,1,\dots,s-1\}$ are all mappings into $\R$.
	\end{definition}

	\noindent To define a logspline family of estimators based on \eqref{exp_fam_dens}, we set the functions $T_j$ to be B-splines, $h$ to be identically 1, $c$ to be the logarithm of the normalization constant and each $\eta_j$ to be the projection on the $j$-th component of the vector valued parameter $y$, as explained in the next definition.
	
	\begin{definition}\label{LogsplineEst}
		Let $T_{N}=(t_{i})_{i=0}^{N}$, $N\in \NN$, be a sequence of knots such that $t_{0}=\dots=t_{k-1}=a$ and $t_{N-k+1}=\dots=t_{N}=b$, where $1\leq k \leq N$, $k$ fixed. Thus, the set of splines $S_{k,T_{N}}$ of order $k$ generated by the B-splines $B_{j,k,T_{N}}$ can be obtained. We suppress the parameters $k, T_{N}$ and just write $B_{j}$ instead of $B_{j,k,T_{N}}$.
		Define the spline function
		\begin{equation*}
		B(\theta;y) = \ds \sum_{j=0}^{J-1}y_{j}B_{j}(\theta)\,, \quad y=(y_{0},\dots,y_{J-1}) \in \R^{J} \quad \text{with} \ J:=N-k+1.
		\end{equation*}
		and for each $y$ we set the probability density function
		\begin{equation}\label{logmodel1}
		\begin{aligned}
		f(\theta;y)&=\exp\Big({\ds\sum_{j=0}^{J-1}y_{j}B_{j}(\theta)-c(y)}\Big)=\exp\Big(B(\theta; y)-c(y)\Big)\,, \\[3pt]
		\text{where} \quad c(y)&=\log\left(\int_{a}^{b}\exp\Big(\sum_{j=0}^{J-1} y_j B_j(\theta)\Big) d\theta \right)<\infty\,.
		\end{aligned}
		\end{equation}
	\end{definition}

	\noindent The family of exponential densities $\{f(\theta;y): \ y\in \R^{J}\}$ is not identifiable since if $\beta$ is any constant, then $c((y_{0}+\beta,\dots,y_{J-1}+\beta))=c(y)+\beta$ and thus
	\begin{equation*}
	f(\theta;(y_{0}+\beta,\dots,y_{J-1}+\beta))=f(\theta;y)
	\end{equation*}
	To make the family identifiable we restrict the vectors $y$ to the set
	\begin{equation*}
	Y_{0}=\left\{ y\in \R^{J} : \sum_{i=0}^{J-1}y_{i}=0 \right\}.
	\end{equation*}
	
	\begin{remark}
		$Y_{0} $ depends only on the number of knots and the order of the B-splines and not the number of samples.
	\end{remark}
	
	\begin{definition}
		We define the logspline model as the family of estimators
		\begin{equation*}
		\mathcal{L}=\big\{f(\theta;y) \ \text{given by \eqref{logmodel1}}: \ y\in Y_{0}\big\}.
		\end{equation*}
		For any $f\in \mathcal{L}$
		\begin{equation*}
		\log{(f)}=\sum_{j=0}^{J-1}y_{j}B_{j}(\theta)-c(y)\in S_{k,T_{N}}.
		\end{equation*}
	\end{definition}
	
	Next, let us pick a set of independent, identically distributed random variables
	\[
	\Theta_n= \big( \theta_{1},\theta_{2},...,\theta_{n} \big) \in \R^n, \ n\in \NN
	\]
	where each $\theta_i$  is drawn from a distribution that has density $p(\theta)$.
	
	\noindent We next define the log-likelihood function $l_n:\R^{J+n} \to \R$ corresponding to the logspline model by
	\begin{equation}\label{logl1}
	\begin{aligned}
	l_n(y) &= l_n(y;\theta_1,\theta,\dots,\theta_n) = l_n(y;\Theta_n) \\
	& =\sum_{i=1}^{n}\log(f(\theta_{i};y))=\sum_{i=1}^{n}    \bigg(\sum_{j=0}^{J-1}y_{j}B_{j}(\theta_i)\bigg) -nc(y)\,, \quad y \in Y_0
	\end{aligned}
	\end{equation}
	and the maximizer of the log-likelihood $l_{n}(y)$ by
	\begin{equation}\label{argmaxlhood}
	\yhat_n= \yhat_n(\theta_1,\dots,\theta_n)= \ds \arg \max_{y\in Y_{0}}l_{n}(y)
	\end{equation}
	whenever this random variable exists, which will be shown on a subset of the sample space whose probability will tend to 1. The density $f( \, \cdot \, ;\yhat_{n})$ is called the \textit{logspline density estimator} of $p$.
	
	We define the expected log-likelihood function  $\lambda_{n}(y)$ by
	\begin{equation}\label{explog1}
	\lambda_{n}(y)=\EE[l(y;\theta_1,\dots,\theta_{n})]=n\left(-c(y)+\int_{a}^{b}\bigg(\sum_{j=0}^{J-1}y_{j}B_{j}(\theta)\bigg)p(\theta) \, d\theta \right) <\infty\,,  \quad y \in Y_{0}.
	\end{equation}
	It follows by a convexity argument that the expected log-likelihood function
	has a unique maximizing value
	\begin{equation*}
	\ybar=\ds \arg \max_{y\in Y_{0}}\lambda_{n}(y)=\ds \arg \max_{y\in Y_{0}}\frac{\lambda_{n}(y)}{n}
	\end{equation*}
	which is independent of $n$ but depends on the knots.		
	
	\par
	
	Note that the function $\lambda_{n}(y)$ is bounded above and goes to $-\infty$ as $|y| \to \infty$ within $Y_0$ and  therefore, due to Jensen's Inequality, the constant $\bar{y}$ is finite; see Stone \cite{Stone90}.  The  estimator $\yhat(\theta_1,\dots,\theta_n)$, in general does not exist. This motivates us to define the set
	\begin{equation}\label{omegan}
	\Omega_n = \bigg\{ \omega \in \Omega:  \yhat=\yhat(\theta_1,\dots,\theta_n) \in \R^{J} \;\; \text{exists} \bigg\}.
	\end{equation}
	In what follows we will show that $\PP(\Omega_n) \to 1$ as $n \to \infty$.  We also note that due to convexity of $l_n(y)$ and $\lambda_n(y)$ the estimators  $\yhat$ and $\ybar$ are unique whenever they exist.

	We define the logspline estimator $\phat$ of $p$ on the space $\Omega_n$ by
	\begin{equation}\label{phatdef}
	\phat: \R\times \Omega_n \quad \text{defined by} \quad \phat(\theta,\omega)=f(\theta, \yhat(\theta_1,\dots,\theta_n)), \ \omega \in \Omega_n
	\end{equation}
	and  define the function
	\begin{equation}\label{pbardef}
	\pbar(\theta):=f(\theta,\ybar)\,,
	\end{equation}
	called an {\it intermediate} estimator of $p$.
	
	\begin{remark}
		In order for the maximum likelihood estimates to be reliable, we require that the modeling error tend to 0 as $n\to \infty$, as described in hypothesis \eqref{hyp2:Lchoice}.
	\end{remark}
	
	\subsection{Notions of distance from the set of splines $S_{k,t}$}
	
	It is a well known fact that continuous functions can be approximated by polynomials. The set of splines $S_{k,t}$ has been introduced in the appendix (Definition \ref{Bsplinespace}) and from what we have stated in remark \ref{Bsplinedense}, that $\bigcup_{N\in \NN}S_{k,T_{N}}$ is dense in the space of continuous functions, there is a question that arises at this point:
	
	Given an arbitrary continuous function $g$ on $[a,b]$, an integer $k\geq 1$ and a set of knots $T_{N}=(t_{i})_{i=0}^{N}$ as in Remark \ref{Bsplinedense}, how close is $g$ to the set $S_{k,T_{N}}$ of splines of order $k$?
	
	We would like to have a bound for the sup-norm distance between $g\in C[a,b]$ and $S_{k,T_{N}}$, where this distance is denoted by $dist(g,S_{k,T_{N}})$ and is defined as
	\begin{equation*}
	\dist(g,S_{k,T_{N}})=\inf_{s\in S_{k,T_{N}}}\|g-s\|_{\infty}, \quad g\in C[a,b].
	\end{equation*}
	The answer to our question is given by Jackson's Theorem found in de Boor \cite{de Boor}. To state it we first need the following definition.
	
	\begin{definition}\label{modcont}
		The modulus of continuity $\omega(g;h)$ of some function $g\in C[a,b]$ for some positive number $h$ is defined as
		\begin{equation*}
		\omega(g;h)=\max\{|g(\theta_1)-g(\theta_2)|: \ \theta_1, \theta_2\in [a,b], \ |\theta_1-\theta_2|\leq h\}.
		\end{equation*}
	\end{definition}

	\noindent The bound given by Jackson's Theorem contains the modulus of continuity of the function whose sup-norm distance we want to estimate from the set of splines. The theorem is stated below.
	
	\begin{theorem}\label{Jackson}
		Let $T_{N}=(t_{i})_{i=0}^{N}$, $N\in \NN$, be a sequence of knots such that $t_{0}=\dots=t_{k-1}=a$ and $b=t_{N-k+1}=\dots=t_{N}$, where $1\leq k\leq N$. Let $S_{k,T_{N}}$ be the set of splines as in definition \ref{Bsplinespace} for the knot sequence $T_{N}$. For each $j\in \{0,\dots,k-1\}$, there exists $C=C(k,j)$ such that for $g\in C^{j}[a,b]$
		\begin{equation*}
		dist(g,S_{k,T_{N}})\leq C \ h^{j} \ \omega\left( \frac{d^{j}g}{d\theta^{j}};|t| \right) \quad \text{where} \quad h=\max_{i}|t_{i+1}-t_{i}|.
		\end{equation*}
		In particular, from the Mean Value Theorem it follows
		\begin{equation}\label{Jacksonbound}
		dist(g,S_{k,T_{N}})\leq C \ h^{j+1} \ \left\| \frac{d^{j+1}g}{d\theta^{j+1}} \right\|_{\infty}
		\end{equation}
		in the case that $g\in C^{j+1}[a,b]$.
	\end{theorem}
	
	\begin{remark}
		Please note that for the approximation the mesh size enters into the bound in \eqref{Jacksonbound} which dictates the placement for the knots.
	\end{remark}

Jackson's Theorem supplies us with an error bound for the approximation of a continuous function by splines. In this paper we are interested in estimates for probability densities. For this reason we state the results specifically for densities provided by Stone \cite{Stone89}.
	
\begin{definition}\label{def::densfamily}
Suppose that $p$ is a continuous probability density supported on some interval $[a,b]$, similar to the set-up when we defined the logspline density estimation method. Define the families of probability densities 
\begin{equation*}
\begin{aligned}
\mathcal{F}_{p}&=\left\{ p_{\alpha}: \ p_{\alpha}(x)=\frac{(p(x))^{\alpha}}{\int (p(y))^{\alpha}\ dy}, \ 0\leq \alpha \leq 1 \right\}\\
\mathcal{F}_{p}^{log} & = \{ \log{(u)}: \ u\in \mathcal{F}_{p} \}.
\end{aligned}
\end{equation*}

\end{definition}

	\begin{lemma}\label{equifamily}
		The family $\mathcal{F}_{p}^{log}$ in Definition \ref{def::densfamily} is equicontinuous  on the set $\{ \theta: \ p(\theta)>0 \}$.
	\end{lemma}
	
	\begin{proof}
		Pick $\epsilon>0$. There exists $\delta>0$ such that $|\log{(p(x))}-\log{(p(y))}|<\epsilon$ whenever $|x-y|<\delta$. Pick any $\alpha\in [0,1)$. If $\alpha=0$ then $p_{0}$ is just a constant and thus $|\log{(p_{0}(x))}-\log{(p_{0}(y))}|=0<\epsilon$. If $0<\alpha<1$, then $|\log{(p_{\alpha}(x))}-\log{(p_{\alpha}(y))}|=|\alpha\log{(p(x))}-\alpha\log{(p(y))}|<\alpha \ \epsilon<\epsilon$.
	\end{proof}
	
	\begin{remark}
		It is practical to work with $p(x)>0$ on the set $[a,b]$ and this is what we assume until the end of the manuscript. In this case, $\log{(p)}\in C[a,b]$.
	\end{remark}
	
	\begin{remark}
		We will be using the notation $\bar{h}=\max_{i}|t_{i+1}-t_{i}|$ and $\underline{h}=\min_{i}|t_{i+1}-t_{i}|$, and $\gamma(T_{N})=\bar{h}/\underline{h}$.
	\end{remark}
	
	We can apply the logspline estimation method to $p$. Let $\pbar$ be defined as in \eqref{pbardef}, the density estimate given by maximizing the expected log-likelihood. We then have the following lemma:
	
	\begin{lemma}\label{pbarsupnormbound}
		Suppose $p$ is a continuous density function supported on $[a,b]$ and $\pbar$ is as in \eqref{pbardef}. Then there exists constant $R=R(\mathcal{F}_{p},k,\gamma(T_{N}))$ that depends on the family $\mathcal{F}_p$, order $k$ and global mesh ratio $\gamma(T_{N})$ of $S_{k,T_{N}}$ such that
		\begin{equation*}
		\|\log{(p)}-\log{(\pbar)}\|_{\infty}\leq R \ \dist(\log({p}),S_{k,T_{N}})
		\end{equation*}
		and therefore
		\begin{equation*}
		\|p-\pbar\|_{\infty}\leq \big(\exp\{R \ \dist(\log({p}),S_{k,T_{N}})\}-1\big)\|p\|_{\infty}.
		\end{equation*}
		Moreover, if $\log({p})\in C^{j+1}([a,b])$ for some $j\in \{0,\dots,k-1\}$ then by Jackson's Theorem we obtain
		\begin{equation*}
		\begin{aligned}
		\|\log{(p)}-\log{(\pbar)}\|_{\infty}&\leq R \ C(k,j) \ \bar{h}^{j+1} \ \left\| \frac{d^{j+1}\log({p})}{d\theta^{j+1}} \right\|_{\infty} \\
		\|p-\pbar\|_{\infty}&\leq \left(\exp\left\{R \ C(k,j) \ \bar{h}^{j+1} \ \left\| \frac{d^{j+1}\log({p})}{d\theta^{j+1}} \right\|_{\infty}\right\}-1\right)\|p\|_{\infty}.
		\end{aligned}
		\end{equation*}
	\end{lemma} 
	
	\begin{remark}
		Please note that the constant $R$ does not depend on the dimension of $S_{k,T_{N}}$. For all practical purposes, we will be using uniformly placed knots, thus suppressing the dependence on $\gamma(T_{N})$, which will be equal to the constant 1.
	\end{remark}
	
	Now we will present certain error bounds required to calculate a bound for {\MISE}. Assume $p,\hat{p}$ and $\pbar$ are as in subsection \ref{subsec:log_fam}. Also assume that $n$ is the number of random samples drawn from $p$.
	
	We will state a series of definitions and theorems that encompass the results from  Lemma 5, Lemma 6, Lemma 7, and Lemma 8 in the work of  Stone\cite{Stone90}[p.728-729].
	\begin{definition}
		Let $y \in Y_{0}$ and $f\in \mathcal{L}$. Let $l_n$ and $\lambda_{n}$ be defined as in \eqref{logl1} and \eqref{explog1}, respectively. For any $n \geq 1$ and $b>0$ we define the set
		\begin{equation*}\label{boundset}
		\begin{aligned}
		A_{n,b}(y) & =\Bigg\{\omega\in \Omega: \left|l(y;\Theta_n(\omega))-l(\ybar;\Theta_n(\omega))-(\lambda_{n}(y)-\lambda_{n}(\ybar))\right| \\
		&\qquad \qquad \qquad <
		nb \bigg( \int | \log(f(\theta;y)) - \log(f(\theta; \ybar)) |^2 d \, \theta \bigg) ^{1	/2}\Bigg\}\,.
		\end{aligned}
		\end{equation*}
	\end{definition}
	
	\par\smallskip
	
	\begin{definition}
		Given $n \geq 1$ and $\epsilon>0$ we define $E_{\epsilon,n}$ to be the subset of ${\cal{F}}=\{ f(\cdot \ ;y) : y\in Y_{0} \}$ such that
		\begin{equation*}
		E_{\epsilon,n}=\left\{ f(\cdot \ ;y) :y\in Y_{0} \ \text{and} \ \bigg( \int | \log(f(\theta;y)) - \log(f(\theta; \ybar)) |^2 d \, \theta \bigg) ^{1/2} \leq n^{\epsilon}\sqrt{\frac{J}{n}} \right\}\,.
		\end{equation*}
	\end{definition}
	
	\par\smallskip

	\begin{lemma}[{Stone\cite{Stone90}[p.728]}] For each $y_1,y_2 \in Y_0$ and $\omega\in \Omega$ we have
		\begin{equation*}
		\left|l(y_1;\Theta_n(\omega))-l(y_2;\Theta_n(\omega))-(\lambda_{n}(y_1)-\lambda_{n}(y_2))\right| \leq
		2n\| \log{f(\cdot \ ;y_1)}-\log{f(\cdot \ ;y_2)} \|_{\infty}\,.
		\end{equation*}
	\end{lemma}

	\par\smallskip
	
	\begin{lemma}[{Stone\cite{Stone90}[p.729]}]
		Let $n \geq 1$. Given $\epsilon>0$ and $\delta>0$, there exists an integer $N=N(n)>0$ and sets $E_j \subset {\cal{F}}$, $j=1,\dots,N$ satisfying
		\[
		\sup_{f_1,f_2 \in E_j} \|\log(f_1)-\log(f_2)\|_{\infty} \leq \delta n^{2 \epsilon - 1}J
		\]
		such that $E_{\eps,n} \; \subset \; \bigcup_{i=1}^{N}E_i$.
	\end{lemma}

	Combining the above lemmas it leads to the following theorem, which is a result outlined in lemmas 5 and 8 found in Stone\cite{Stone90}.
	
	\begin{theorem} \label{Aboundthm}
		Given $D>0$ and $\epsilon>0$, let  $b_{n}=n^{\epsilon}\sqrt{\dfrac{J}{n}}$, $n \geq 1$, and $0<\epsilon< \frac{1}{2}$ and $\beta=\epsilon$ as in \eqref{hyp2:Lchoice}. There exists $N=N(D)$ such that for all $n>N$
		\begin{equation*}
		A_{n,b_n}(y) \subset \Omega_n \quad \text{for each $y \in Y_0$}
		\end{equation*}
		and thus
		\begin{equation}\label{Acest}
		\PP(\Omega_{n}^{c})\leq \PP\big(A^c_{n,b_{n}}(y)\big)\leq 2e^{{-n^{2\epsilon}J\delta(D)}} \,.
		\end{equation}
		
	\end{theorem}
	
	\begin{remark}
		From \eqref{Acest} we can see that as number of samples goes to infinity, we have that \[\PP(\Omega_{n})\to 1 \ \text{as} \ n\to \infty.\]
	\end{remark}
	
	\bigskip
	\begin{remark}
		The bound \eqref{Acest} presented in Theorem \ref{Aboundthm} is a consequence of Hoeffdings inequality which states that for any $t>0$
		\[
		\PP \bigg( \Big| \frac{1}{n} \sum_{i=1}^n X_i  - \EXP X_1 \Big| \geq t \bigg) \leq 2 \exp\bigg(
		-\frac{2 n^2 t^2}{\sum_{i=1}^n (b_i-a_i)^2}  \bigg)
		\]
		where $X_1,\dots,X_n$ are identically distributed independent random variables with $\PP(X_1 \in [a_i,b_i])=1$. To get the bound \eqref{Acest} one needs to choose
		\[
		t = b \bigg( \int | \log(f(\theta;y)) - \log(f(\theta; \ybar)) |^2 \, d\theta \bigg)^{\frac{1}{2}}.
		\]
	\end{remark}
	
	Now that we have defined the set where $\yhat$ exists and showed that the probability of its complement vanishes as $n\to \infty$ with a specific exponential rate, we will now state certain rates of convergence that only apply on $\Omega_{n}$. The following theorem contains results of Theorem 2 and Lemma 12 of Stone\cite{Stone90}.

	\begin{theorem}\label{estonOmegan}
		There exist constants $R_3, R_4, R_5$ and $R_6$ such that for all $\omega\in \Omega_{n}$
		\begin{equation*}
		\begin{aligned}
		& \| \log{\hat{p}(\cdot,\omega)}-\log{\pbar(\cdot)}\|_{\infty}  \leq R_{3}\dfrac{J}{\sqrt{n}}\\
		& \| \hat{p}(\cdot,\omega)-\pbar(\cdot)\|_{2}  \leq R_{4}\sqrt{\dfrac{J}{n}}\\
		& \| \hat{p}(\cdot,\omega)-\pbar(\cdot)\|_{\infty}  \leq R_{5}\sqrt{\dfrac{J\log{(J)}}{n}}\\
		& \| \log{\hat{p}(\cdot,\omega)}-\log{\pbar(\cdot)}\|_{2}  \leq R_{6}\sqrt{\dfrac{J}{n}}
		\end{aligned}
		\end{equation*}
	\end{theorem}

	\section{Parallel logspline estimators}\label{section:parest} In what follows we introduce the parallel estimators that allow us to carry out our analysis and construct an estimator for the full dataset posterior. This approach enables one to analyze datasets that are partitioned and stored in different locations. The estimators are constructed using logspline density estimation, as presented in \cite{Miroshnikov-Conlon}.

	\begin{enumerate}
		\renewcommand{\labelenumi}{\textbf{(\theenumi)}}
		\renewcommand{\theenumi}{H\arabic{enumi}}
		\setcounter{enumi}{-1}
		\item\label{hyp0:model} Motivated by the approaches introduced in \cite{Neiswanger,Wang,Miroshnikov-Savelev}, we assume that the full dataset posterior $p(\theta|\textbf{x})$, where $\theta \in \R$ is the model parameter and $\textbf{x} = \{\textbf{x}_1, \textbf{x}_2,\dots, \textbf{x}_M\}$ is the full data set partitioned into $M$ disjoint independent subsets, is reconstructed by estimating posterior densities $p(\theta|\textbf{x}_{m})$ as follows 
	\begin{equation*}
	p(\theta|\textbf{x})\propto \ds\prod_{m=1}^{M}p_{m}(\theta)=: p^*(\theta)\,, \quad \text{where} \quad p_{m}(\theta):=p(\theta|\textbf{x}_{m})=p(\theta)^{1/M}p(\textbf{x}_m | \theta).
	\end{equation*}
	Here we assume that $p_m(\theta), \ m\in \{1,\dots,M\}$ have compact support on the interval $[a,b]$. Furthermore, we assume that $\|p^*\|_2^2 < \infty$.
	\end{enumerate}

	\begin{definition}\label{def::parestform}
	\begin{enumerate}
		\item[]
		\item 	Let $\theta_1^m, \theta_2^m, \dots,\theta^m_{n_m} \sim p_m(\theta)$ be i.i.d.~random variables. Let $K_m$ be the number of knots and $k_m$ the order of the B-splines. Define the sample subspace	
		\begin{equation}\label{omeganm}
		\Omega_{n_m}^{m} := \bigg\{ \omega \in \Omega:  \yhat=\yhat(\theta_1^{m},\dots,\theta_{n_m}^{m}) \in \R^{J_m} \;\; \text{exists} \bigg\}.
		\end{equation}
		where $J_m:=K_m-k_m+1$ and $\hat{y}$ is defined  in \eqref{argmaxlhood}.
		\item Define the estimator of $p$ to be of the form
		\begin{equation*}
		\hat{p}(\theta)\propto
		\ds \hat{p}^*(\theta) \quad \text{where} \quad \hat{p}^*(\theta):=\prod_{m=1}^{M}\hat{p}_m(\theta)
		\end{equation*}
		and  $\hat{p}_m(\theta)$ is the logspline density estimator of  $p_m(\theta)$ that has the form
		\begin{equation*}
	\phat_{m}(\theta)=\exp\left( B_{m}(\theta;\yhat^{m})-c(\yhat^{m}) \right)
	\end{equation*}
	where
	\begin{equation*}
	B_{m}(\theta;\yhat^{m})=\sum_{j=0}^{J_{m}(n)-1}\yhat_{j}^{m}B_{j,k,T_{K_{m}(n)}}(\theta)
	\quad \text{and} \quad c(\yhat^{m})=\log\left( \int \exp\left( B_{m}(\theta;\yhat^{m})\, d\theta \right) \right),
	\end{equation*}
	and $\yhat^{m}=(\yhat_{1}^{m},\dots,\yhat_{J_{m}(n)-1}^{m})$ maximizes the log-likelihood, as described in equation \eqref{logl1}.
\item 	We also define the intermediate estimator $\bar{p}^*$ of $p$ by
	\[
	\pbar^*(\theta):=\prod_{m=1}^{M}\pbar_m(\theta)
	\]
	where $\pbar_m$ is the intermediate estimator of $p_m$ as defined in \eqref{pbardef}.
	\end{enumerate}
	\end{definition}

	\begin{definition}
	The mean integrated square error of the estimator $\hat{p}^*$ is defined by
	\begin{equation*}
	 \MISE(p^*, \hat{p}^*; \N) =  \EXP \int (\hat{p}^*(\theta;\omega) - p^*(\theta))^2 \, d\theta\,\quad
	\end{equation*}
	where $\N = (n_1,n_2,\dots,n_M)$, with $n_m$ being the number of samples for the $m$-th posterior. Given the subset $\tilde{\Omega} \subset \Omega$, the conditional mean integrated square error is defined by
	\begin{equation*}
	 \MISE(p^*, \hat{p}^*; \N \, | \tilde{\Omega}) =  \EXP_{\tilde{\Omega}} \int (\hat{p}^*(\theta;\omega) - p^*(\theta))^2 \, d\theta
	\end{equation*}
	where $\EXP_{\tilde{\Omega}}$ denotes the conditional expectation given the event $\tilde{\Omega}$.
	\end{definition}

\begin{remark} \rm \label{MCMCIID}
In our work, the asymptotic analysis is performed under the assumption that the samples drawn from each subset posterior distribution are i.i.d. It is well-known, however, that samples produced with MCMC methods always have a degree of autocorrelation present in them. For Bayesian kernel density estimators, for example, the convergence analysis in the MCMC settings by \cite{De Valpine, West1993, Skold2003} shows that such estimators are asymptotically exact under certain conditions. While the autocorrelation might affect the convergence of logspline estimators, the analysis of this phenomenon is outside the scope of this paper. Meanwhile the numerical experiments in the later sections illustrate there is little effect of MCMC-produced samples on convergence.  

Finally, we note that there are several techniques available that can reduce the dependence between samples obtained with MCMC methods. One can run independent Markov chains for each sample, discard a number of intermediate samples between the recorded samples, or employ so-called perfect sampling \cite{Propp}, which guarantees i.i.d samples.

\end{remark}

	\section{Main results}\label{section:main_res}
	
	\subsection{Error bound for the unnormalized estimator}
	
	Suppose we are given a data set $\textbf{x}$ and it is partitioned into $M\geq 1$ disjoint subsets $\textbf{x}_{m}, \ m\in \{1,\dots,M\}$. We are interested in the subset posterior densities $p_{m}(\theta)=p(\theta |\textbf{x}_{m})$. For each such density we apply the analysis from before. Let $\phat_{m}$ and $\pbar_{m}, \ m\in \{1,\dots,M\}$ be the corresponding logspline estimators as defined in \eqref{phatdef} and \eqref{pbardef} respectively. By definition of $\phat_{m}$, that is equal to the logspline density estimate on $\Omega_{n_m}^{m}\subset \Omega$, where $\Omega_{n_m}^{m}$ is the set defined in \eqref{omegan} for $\phat_{m}$.
	
	\begin{definition}\label{interOmega}
		For $m\in\{1,\dots,M\}$, let $\Omega_{n_m}^{m}$ be the set defined in \eqref{omeganm}. We then set
		\[
		\Omegaline^{M,\N}:=\bigcap_{m=1}^{M}\Omega_{n_m}^{m} \quad \text{where} \quad \N=(n_1,\dots,n_m)
		\]
		which is the set where the maximizer for the log-likelihood exists given each data subset and thus all logspline density estimators $\phat_{m}$ exist.
	\end{definition}
	
	Our first result is the asymptotic bound for $\MISE$ of the unnormalized densities $p^{*}$ and $\phat^{*}$.
	\begin{theorem} \label{thm::bound_unnorm_est}
		Assume the hypotheses \eqref{hyp0:model}-\eqref{hyp4:pcond1} hold and $M\geq 1$. 

		\begin{itemize}

			\item [(a)] The sample subspace in Definition \ref{interOmega} satisfies
			\[
			\lim_{n\to\infty}\PP\Big(\Omegaline^{M,\N(n)}\Big)=1.
			\]
			\item [(b)]  The following bound holds for large $n$
		\begin{equation*}
		\begin{aligned}
		&\mkern-18mu	\MISE(p^{*},\phat^{*} \ | \ \underline{\Omega}^{M,\N(n)})\\
		&\leq \left[R_1 \sum_{m=1}^{M}\sqrt{\dfrac{J_m(n)}{N_m(n)}}+\left( \exp\left\{R_2 \ \bar{h}_{max}^{q} \ \sum_{m=1}^{M}\left\| \frac{d^{q}\log({p_m})}{d\theta^{q}} \right\|_{\infty} \right\}-1 \right)\|p^{*}(\theta)\|_2\right]^{2}
		\end{aligned}
		\end{equation*}
		for some constants $R_{1},R_{2}$.

		\item [(c)] Suppose \eqref{hyp6:hmaxcond} holds. Then
		\begin{equation*}
		\sqrt{\MISE(p^{*},\phat^{*})}=O(M n^{-\bar{\beta}(\alpha)})=O(M^{1-\bar{\beta}(\alpha)}\| \N(n) \|^{-\bar{\beta}(\alpha)})
		\end{equation*}
where 
\[
\bar{\beta}(\alpha)=\min\left(\frac{1-\alpha}{2},\alpha q\right) 
\]
Thus the optimal rates for the number of knots and the error rate are given by
\[
\alpha_{opt}= \frac{1}{2q+1}, \quad \beta_{opt}=\bar{\beta}(\alpha_{opt})=\frac{q}{2q+1}.
\]

		
		\end{itemize}
	
	\end{theorem}
		
	\vspace{5pt}	
	\begin{remark}
	 We do the analysis of $\MISE$ on the sample subspace $\Omegaline^{M,\N(n)}$ because the estimators $\{\phat_m\}_{m=1}^M$ exist there. However, the probability of the set where the estimators $\phat_m$ exist tends to $1$ and, in practice, one would never encounter the set where the maximizer of the log-likelihood does not exist.
	\end{remark}

	\begin{remark}
		It's interesting to note how the number of knots, their placement and the number of samples all play a role in the above bound. If we want to be accurate, all of the parameters $J_{m}(n), N_{m}(n)$ and $\bar{h}_{max}$ must be chosen appropriately.
	\end{remark}
	
	\subsection{Error bound for the normalized estimator}\label{AnalysisRenormConst}
	
	We next consider the error bound for $\MISE$ when one renormalizes the product of the estimators so it can be a probability density. The renormalization can affect the error since $p^{*}$ and $\phat^{*}$ are rescaled. We define the renormalization constant and its estimator to be
	\begin{equation}\label{lambdahat}
	\lambda:=\int p^{*}(\theta)\,d\theta \quad \text{and} \quad \lamhat(\omega):=\int \phat^{*}(\theta;\omega)\,d\theta
	\end{equation}
	and their reciprocals by $c:=\lambda^{-1}$ $\hat{c}:=\lamhat^{-1}$. We now state our main result.

	\begin{theorem}\label{MISEpphat} Let $M\geq 1$ and suppose hypotheses \eqref{hyp0:model}-\eqref{hyp6:hmaxcond} hold.

\begin{enumerate}

\item [(a)]  Let $\lambda$ and $\lamhat(\omega)$ be defined as in \eqref{lambdahat}. Then for $n$ large enough we have
		\begin{equation*}
		\left|\frac{\lamhat(\omega)}{\lambda}-1\right|=O(M^{1-\bar{\beta}(\alpha)}\| \N(n) \|^{-\bar{\beta}(\alpha)}), \quad  \omega \in \Omegaline^{M,\N(n)}.
		\end{equation*}

\item [(b)] Furthermore, we have
\begin{equation*}
		\MISE\left( p,\phat \ | \ \underline{\Omega}^{M,\N(n)} \right)=O(M^{2-2\bar{\beta}(\alpha)}\|\N(n)\|^{-2\bar{\beta}(\alpha)}).
		\end{equation*}
	The optimal rate can then be obtained by setting $\bar{\beta}(\alpha)=\beta_{opt}$.
\end{enumerate}
\end{theorem}

\begin{remark}
	Part (a) of the above theorem  suggests that the renormalization constant $\hat{c}$ of the product of the estimators approximates the true renormalization constant $c$, when $\omega$ is restricted to the sample subspace $\underline{\Omega}^{M,\N(n)}$, the space where the logspline density estimators $\phat_{m}$, $m\in \{1,\dots,M\}$ are well-defined.
	Given the number of partitions $M$ of the data, part (2) showcases the asymptotic behavior of $\MISE$ optimally with respect to the number of samples per partition and the number of partitions.
\end{remark}	
	
	
	\subsection{Error bound for interpolated estimator}
In practice it is difficult to evaluate the renormalization constant 
	\[
	\lamhat(\omega)=\int\phat^{*}(\theta)\,d\theta= \int\prod_{m=1}^{M}\phat_{m}(\theta)\,d\theta
	\]
	defined in \eqref{lambdahat}. The difficulty is due to the process of generating MCMC samples and thus $\phat^{*}$ is not explicitly known. In order to circumvent this issue, our idea is to approximate the integral above numerically. To accomplish this, we interpolate $\phat^{*}$ using Lagrange polynomials. This procedure leads to the construction of an interpolant estimator $\ptil^{*}$ which we then integrate numerically. We then normalize $\ptil^{*}$ and use that as a density estimator for $p$. Unfortunately, to estimate the error by considering that kind of approximation given an arbitrary grid of points for Lagrange polynomials, independent of the set of knots $(t_{i})$ for B-splines gives a stringent condition on the smoothness of B-splines we incorporate. It turns out that we have to utilize B-splines of order at least $k=4$. For this reason we consider using Lagrange polynomials of order $l+1$ which satisfy $l<k-2$.
	
By Remark \ref{CurryScho}, in Section \ref{section:appendix}, we have that B-splines of order $k$, and therefore any splines that arise from these, will have $k-2$ continuous derivatives on $(a,b)$. Thus, in order to utilize Lemma \ref{estinterrlmm}, we must have that the order of the Lagrange polynomials be at most $k-2$, i.e. $l\leq k-3$. Since $l\geq 1$ this implies that the B-splines used in the construction of the logspline estimators be at least cubic. Thus, assume $k\geq 4$ and let $1\leq l\leq k-3$ be a positive integer that denotes the degree of the interpolating polynomials. Let $N\in \NN$ be the number of sub-intervals of $[a,b]$ on each of which we will interpolate the product of
	estimators by the polynomial of degree $l$. Thus each sub-interval has to be further  subdivided into $l$
	intervals. Define the partition $\mathcal{X}$ of $[a,b]$ such that
	\begin{equation}\label{meshX}
	\mathcal{X} = \{a=x_0 < x_1 < x_2 < \dots < x_{Nl}=b \} \,  \quad \text{and} \quad x_{i+1}-x_i = \frac{b-a}{Nl} = \Delta x.
	\end{equation}
	For each $i=0,\dots, N-1$, recalling the formula \eqref{lagrp}, we define the (random) Lagrange polynomial
	\begin{equation*}
	\hat{q}_i(\theta) := \sum_{\tau=0}^l \phat^{*}(x_{il+\tau}) l_{\tau,i}(\theta)
	\quad \text{with} \quad
	\hspace{0.2 in}l_{\tau,i}(\theta):=\ds\prod_{j \in \{0, \dots, l\}\backslash\{\tau\}}
	\left(\frac{\theta-x_{il+j}}{x_{il+\tau}-x_{il+j}}\right)\,,
	\end{equation*}
	which is a polynomial that interpolates the estimator $\phat^{*}(\theta)$ on the interval
	$[x_{il},x_{(i+1)l}]$. We next define an interpolant estimator
	$\ptil^{*}$ to be a {\it random} composite
	polynomial given by
	\begin{equation}\label{compintp}
	\ptil^{*}(\theta):= \left\{
	\begin{aligned}
	&  0,& & \theta \in \RR \backslash [a,b] \\
	&  \hat{q}_i(\theta),& & \theta \in [x_{il},x_{(i+1)l}]
	\end{aligned}\right.
	\end{equation}
	which approximates the estimator $\phat^{*}$ on the whole interval $[a,b]$. We define the renormalization constant $\tilde{c}$ and the density estimator $\ptil$ of $\phat$ as follows:
	\begin{equation}\label{lamtil}
	\frac{1}{\tilde{c}}=\tilde{\lambda}=\int_{a}^{b}\ptil^{*}(\theta)\, d\theta \quad \text{and} \quad \ptil:=\tilde{c}\ptil^{*}.
	\end{equation} 

	\noindent We now state the main result.

	\begin{theorem}\label{thm::bound_interp_est}

	Assume that hypotheses \eqref{hyp0:model}-\eqref{hyp6:hmaxcond} hold, that $1\leq l\leq k-3$ and the partition $\mathcal{X}$ is as defined in \eqref{meshX}.
	\begin{itemize}
		\item[(a)] The following bound holds for the unnormalized estimator
		\begin{equation}\label{MISEpptilest}
		\MISE (\phat^{*}\,,\ptil^{*} \ | \ \underline{\Omega}^{M,\N(n)}) \leq C \Bigg( \frac{(\Delta x)^{l+1}}{4(l+1)}\|\N(n)\|^{(l+1) \alpha} M^{l+1}  \Bigg)^2
		\end{equation}
		where $C=C(l+1,k,p_{1},\dots,p_{M},(a,b))$.
		\item[(b)] The error bound for the distance between the two renormalization constants is bounded by
		\begin{equation*}
		|\lamhat-\tilde{\lambda}|\leq C\Bigg( \frac{(\Delta x)^{l+1}}{4(l+1)}\|\N(n)\|^{(l+1) \alpha} M^{l+1}  \Bigg)
		\end{equation*}
		where the constant $C=C(l+1,k,p_{1},\dots,p_{M},(a,b))$.
		\item[(c)] The following bound holds for the normalized estimator
		\begin{equation*}
		\MISE\left( \phat,\ptil \ | \ \underline{\Omega}^{M,\N(n)} \right)=O\left[ \Big( \|\N(n)\|^{\alpha} (\Delta x) M  \Big)^{2(l+1)}\right].
		\end{equation*}
	\end{itemize}
	\end{theorem}

	\vspace{6 pt}
	\begin{theorem}\label{MISEestthm}
		Assume that hypotheses \eqref{hyp0:model}-\eqref{hyp6:hmaxcond} hold. Let $\ptil$ be the polynomial that interpolates $\phat$ as defined in \eqref{compintp}, given the partition $\mathcal{X}$.
		We then have the estimate
		\begin{equation}\label{MISEest}
		\MISE (p\,,\ptil \ | \ \underline{\Omega}^{M,\N(n)}) \, \leq \, C\left[ M^{2-2\bar{\beta}(\alpha)}\|\N(n)\|^{-2\bar{\beta}(\alpha)}+\Bigg((\Delta x)\|\N(n)\|^{\alpha} M  \Bigg)^{2(l+1)} \right]
		\end{equation}
		where the constant $C$ depends on the order $k$ of the B-splines, the degree $l$ of the interpolating polynomial, the densities $p_{1},\dots,p_{M}$ and the length of the interval $(a,b)$. Furthermore, assuming that $\Delta x$ is a function of the vector of samples $\N(n)$, then $\MISE$ scales optimally with respect to $\N(n)$ such that
		\begin{equation}\label{MISEestthm_b}
		\MISE (p\,,\ptil\ | \ \underline{\Omega}^{M,\N(n)}) \quad = \,  O(M^{2-2\beta_{opt}}\|\N(n)\|^{-2\beta_{opt}}) \quad \text{when} \quad \Delta x = O\bigg( \|\N(n)\|^{-\beta_{opt}\left(\frac{1}{l+1}+\frac{1}{q}\right)}\bigg) \,.
		\end{equation}
	\end{theorem}

\section{Proof of error estimate theorems}\label{section::proofs}

\subsection{Proof of the error bound for the unnormalized estimator}

\noindent{\bf Proof of Theorem \ref{thm::bound_unnorm_est}}
	\begin{proof} 
		By Theorem \ref{Aboundthm} we have that
		\[
		\begin{aligned}
		&\PP\left( \Omega \setminus \Omegaline^{M,\N(n)} \right)=\PP\left( \bigcup_{m=1}^{M}(\Omega_{N_m(n)}^{m})^{c} \right) \\
		&\leq \sum_{m=1}^{M}\PP\left( (\Omega_{N_m(n)}^{m})^{c} \right)\leq \sum_{m=1}^{M} 2 e^{{-N_m(n)^{2\epsilon}J_m(n)\delta_m(D)}}.
		\end{aligned}
		\]
	Sending $n$ to infinity yields part (a). To prove the bound for the conditional $\MISE$ in (b), we make the following claim:
	
	\noindent{\bf Claim:} Let $q \leq k$ be the constant in \eqref{hyp4:pcond1}. We then have
	\begin{itemize}
		\item[(i)] For large $n$, there exists a positive constant $R_1$ such that
		\[
		\|\phat^{*}(\theta;\omega)-\pbar^{*}(\theta)\|_2\le R_1 \sum_{m=1}^{M}\sqrt{\dfrac{J_m(n)}{N_m(n)}}.
		\]
		\item[(ii)] There exists a positive constant $R_{2}=R_2(M,k,q,\mathcal{F}_p,\gamma(T_{K_{1}(n)}),\dots,\gamma(T_{K_{M}(n)}))$ such that
		\[
		\|\log{(p^{*})}-\log{(\pbar^{*})}\|_{\infty}\leq R_2 \ \bar{h}_{max}^{q} \ \sum_{m=1}^{M}\left\| \frac{d^{q}\log({p_m})}{d\theta^{q}} \right\|_{\infty} \quad \text{where} \quad \bar{h}_{max}=\max_{m}\{\bar{h}_{m}\}.
		\]
		\item[(iii)] Using the bounds from $(a)$ and $(b)$ we have
		\[
		\|\phat^{*}(\theta;\omega)-p^{*}(\theta)\|_2\leq R_1 \sum_{m=1}^{M}\sqrt{\dfrac{J_m(n)}{N_m(n)}}+\left( \exp\left\{R_2 \ \bar{h}_{max}^{q} \ \sum_{m=1}^{M}\left\| \frac{d^{q}\log({p_m})}{d\theta^{q}} \right\|_{\infty} \right\}-1 \right)\|p^{*}(\theta)\|_2.
		\]
	\end{itemize}

\noindent {\bf Proof of claim:}
\begin{itemize}
	\item[(i)] Let $M=2$. We have
	\begin{align*}
		\|\phat^{*}(\theta;\omega)-\pbar^{*}(\theta)\|_2&=\|\phat_1 \phat_2-\pbar_1 \pbar_2\|_2=\|\phat_1 \phat_2-(\pbar_1-\phat_1+\phat_1) \pbar_2\|_2\\
		&=\|\phat_1 (\phat_2-\pbar_2)-(\pbar_1-\phat_1)\pbar_2\|_2\\
		&\le \|\phat_1(\phat_2-\pbar_2)\|_2+\|\pbar_2(\pbar_1-\phat_1)\|_2
	\end{align*}
	Observe that
	\[
	\|\phat_1\|_{\infty}\le \|\phat_1-\pbar_1\|_{\infty}+\|p_1-\pbar_1\|_{\infty}+\|p_1\|_{\infty}
	\quad \text{and} \quad \|\pbar_2\|_{\infty}\le \|p_2-\pbar_2\|_{\infty}+\|p_2\|_{\infty}\]
	and then by applying Lemma \ref{pbarsupnormbound} and Theorem \ref{estonOmegan}, for large enough $n$ there exist positive constants $R_4^{(1)}, R_4^{(2)}$ that respectively depend on $\|p_2\|_{\infty}$ and $\|p_1\|_{\infty}$ such that
	\[
	\|\phat^{*}(\theta;\omega)-\pbar^{*}(\theta)\|_2 \le \sum_{m=1}^{2}R_4^{(m)}\sqrt{\dfrac{J_m(n)}{N_m(n)}}.
	\]
	By induction we obtain the result by setting $R_1=\max\{R_4^{(1)},\dots,R_4^{(M)}\}$.
	\item[(ii)] We write
	\begin{align*}
		\|\log(p^{*}(\cdot))-\log(\pbar^{*}(\cdot))\|_{\infty}&=\| \log(\prod_{m=1}^{M}p_m(\cdot))-\log(\prod_{m=1}^{M}\pbar_m(\cdot)) \|_{\infty} \\
		&\leq \sum_{m=1}^{M}\|\log(p_m(\cdot))-\log(\pbar_m(\cdot))\|_{\infty}
	\end{align*}
	and then we apply Lemma \ref{pbarsupnormbound}. For each $m\in \{1,\dots,M\}$ there will be constants $R_{m}$ and $C_{m}(k,j)$ appearing and we can take $R_{2}=\max_{m}\{M'_{m}C_{m}(k,j)\}$.
	\item[(iii)] By the triangle inequality we have
	\begin{equation}\label{l2_pbar_p}
		\|\phat^{*}(\theta;\omega)-p^{*}(\theta)\|_2\le 
		\|\phat^{*}(\theta;\omega)-\pbar^{*}(\theta)\|_2 + \|\pbar^{*}(\theta)-p^{*}(\theta)\|_2.
	\end{equation}
	For the second term on the right-hand side we write
	\[
	\|\pbar^{*}(\theta)-p^{*}(\theta)\|_2 = \left\|p^{*}(\theta)\left(\frac{\pbar^{*}(\theta)}{p^{*}(\theta)}-1\right)\right\|_2 =\left\|p^{*}(\theta)\left(\exp\{\log(\pbar^{*}(\theta))-\log(p^{*}(\theta))\}-1\right)\right\|_2.
	\]
	
	If $\pbar^{*}(\theta)\geq p^{*}(\theta)$ then
	\[
	\left|\exp\{\log(\pbar^{*}(\theta))-\log(p^{*}(\theta))\}-1\right|=\exp\{\log(\pbar^{*}(\theta))-\log(p^{*}(\theta))\}-1.
	\]
	
	If $\pbar^{*}(\theta)< p^{*}(\theta)$ then
	\begin{align*}
		\left|\exp\{\log(\pbar^{*}(\theta))-\log(p^{*}(\theta))\}-1\right|&=1-\exp\{\log(\pbar^{*}(\theta))-\log(p^{*}(\theta))\}\\
		&\leq \exp\{\log(p^{*}(\theta))-\log(\pbar^{*}(\theta))\}-1
	\end{align*}
	where the inequality is justified by the fact that $1-e^{-x}\leq e^{x}-1, \ \text{for any} \ x\geq 0$. This implies
	\begin{align*}
		\|\pbar^{*}(\theta)-p^{*}(\theta)\|_2&\leq \left\|p^{*}(\theta)\left(\exp\{|\log(\pbar^{*}(\theta))-\log(p^{*}(\theta))|\}-1\right)\right\|_2\\
		&\leq
		\left(\exp\{\|\log(\pbar^{*}(\theta))-\log(p^{*}(\theta))\|_{\infty}\}-1\right)\left\|p^{*}(\theta)\right\|_2
	\end{align*}
	The result is obtained by applying the bounds from parts (i) and (ii) to \eqref{l2_pbar_p}. This concludes the proof of the claim.
\end{itemize}
	
	Thus, for the conditional $\MISE$ we have
	\begin{align*}
	&\MISE(p^{*},\phat^{*}\ | \ \underline{\Omega}^{M,\N(n)})=\EE_{\underline{\Omega}^{M,\N(n)}}\int (\phat^{*}(\theta;\omega)-p^{*}(\theta))^{2}\,d\theta\\
	&\leq \EE_{\underline{\Omega}^{M,\N(n)}}\left[R_1 \sum_{m=1}^{M}\sqrt{\dfrac{J_m(n)}{N_m(n)}}+\left( \exp\left\{R_2 \ \bar{h}_{max}^{q} \ \sum_{m=1}^{M}\left\| \frac{d^{q}\log({p_m})}{d\theta^{q}} \right\|_{\infty} \right\}-1 \right)\|p^{*}(\theta)\|_2\right]^{2}
	\end{align*}
	which proves part (b). Finally, if \eqref{hyp6:hmaxcond} holds, then part (c) follows directly.

	\end{proof}

\subsection{Proof of the error bound for the normalized estimator}

\noindent{\bf Proof of Theorem \ref{MISEpphat}}
\begin{proof}
	By definition of $\lambda$ and $\lamhat(\omega)$ and following similar steps as in the proof of Theorem \ref{thm::bound_unnorm_est}, we have
		\begin{align*}
		|\lambda-\lamhat(\omega)|&=\left| \int p^{*}(\theta)\, d\theta-\int \phat^{*}(\theta;\omega)\, d\theta \right|\\
		&\le \int \left(\exp\{|\log(\phat^{*}(\theta))-\log(p^{*}(\theta))|\}-1\right)|p^{*}(\theta)|\,d\theta\\
		&\le \|p^{*}(\theta)\|_{2}\|\exp\{|\log(\phat^{*}(\theta))-\log(p^{*}(\theta))|\}-1\|_{2}
		\end{align*}
	By Lemma \ref{pbarsupnormbound} and Theorem \ref{estonOmegan} we have for $n$ large enough that
	\[
	\|\log{(\phat^*(\cdot,\omega))}-\log{(p^*(\cdot))}\|_{\infty}\le \|\log{(\phat^*(\cdot,\omega))}-\log{(\pbar^*(\cdot))}\|_{\infty}+\|\log{(\pbar^*(\cdot))}-\log{(p^*(\cdot))}\|_{\infty}\le 1
	\]
	which implies that
	\[
	\exp\{|\log(\phat^{*}(\theta))-\log(p^{*}(\theta))|\}-1 = O(|\log(\phat^{*}(\theta))-\log(p^{*}(\theta))|).
	\]
	Thus, there exists some constant $C$ such that
	\begin{align*}
		|\lambda-\lamhat(\omega)| &\leq C\|p^{*}(\theta)\|_{2}\|\log(\phat^{*}(\theta))-\log(p^{*}(\theta))\|_{2}\\
		&\le C\|p^{*}(\theta)\|_{2}(\|\log(\phat^{*}(\theta))-\log(\pbar^{*}(\theta))\|_{2}+\|\log(\pbar^{*}(\theta))-\log(p^{*}(\theta))\|_{2})\\
		&\le C\|p^{*}(\theta)\|_{2}\left( R_1 \sum_{m=1}^{M}\sqrt{\dfrac{J_m(n)}{N_m(n)}}+ \exp\left\{R_2 \ \bar{h}_{max}^{q} \ \sum_{m=1}^{M}\left\| \frac{d^{q}\log({p_m})}{d\theta^{q}} \right\|_{\infty} \right\}-1 \right)
	\end{align*}
		for some constants $R_1,R_2$. Dividing by $\lambda$ and considering \eqref{hyp6:hmaxcond} the result in part (a) follows.

	To analyze $\MISE(p,\phat)$ we introduce an intermediate functional
		\begin{equation*}
		\overline{\MISE}\left(p,\phat \ | \ \underline{\Omega}^{M,\N(n)}\right):=\EE_{\underline{\Omega}^{M,\N(n)}}\left[ \left(\frac{\lamhat(\omega)}{\lambda}\right)^{2}\int (\phat(\theta;\omega)-p(\theta))^{2}\,d\theta \right].
		\end{equation*}
We now show that functional is asymptotically equivalent to $\MISE$, thus providing us with the means to view how $\MISE$ scales without having to directly analyze it. Notice that 
\begin{equation*}
		\begin{aligned}
		\overline{\MISE}\left(p,\phat \ | \ \underline{\Omega}^{M,\N(n)}\right)&=\EE_{\underline{\Omega}^{M,\N(n)}}\left[ \left(\frac{\lamhat}{\lambda}-1+1\right)^{2}\int (\phat(\theta;\omega)-p(\theta))^{2}\,d\theta \right]\\
		&=\EE_{\underline{\Omega}^{M,\N(n)}}\left[ \left[\left(\frac{\lamhat}{\lambda}-1\right)^{2}+2\left(\frac{\lamhat}{\lambda}-1\right)+1\right]\int (\phat(\theta;\omega)-p(\theta))^{2}\,d\theta \right]
		\end{aligned}		
\end{equation*}
		and hence part (a) implies

\begin{equation}\label{intermMISE1}
\begin{aligned}
\overline{\MISE}\left(p,\phat \ | \ \underline{\Omega}^{M,\N(n)}\right)=\big(1+O(M^{1-\bar{\beta}(\alpha)}\|\N(n)\|^{-\bar{\beta}(\alpha)})\big)\MISE\left(p,\phat \ | \ \underline{\Omega}^{M,\N(n)}\right).
\end{aligned}
\end{equation}
Let $\EE_n(\cdot):=\EE(\cdot|\underline{\Omega}^{M,\N(n)})$. Notice that 
	\begin{align*}
	\overline{\MISE}\left( p,\phat \ | \ \underline{\Omega}^{M,\N(n)} \right)
	&=\|p\|_{2}^{2}\ \EE_n\left[\left(\frac{\lamhat}{\lambda}-1\right)^{2}\right]+\lambda^{-2}\ \MISE_n(p^{*},\phat^{*})\\
	&\quad -2\lambda^{-1}\ \EE_n\int\left(\frac{\lamhat}{\lambda}-1\right)(\phat^{*}-p^{*})p\,d\theta\\
	&=J_{1}+J_{2}+J_{3}
	\end{align*}
	We next determine how each $J_{i}$, $i\in \{1,2,3\}$ scales. For $J_1$ part (a) implies
	\[
	J_1=O(M^{2-2\bar{\beta}(\alpha)}\|\N(n)\|^{-2\bar{\beta}(\alpha)}),
	\]
	while for $J_2$ Theorem \ref{thm::bound_unnorm_est}(c) implies
	\[
	J_2=O(M^{2-2\bar{\beta}(\alpha)}\|\N(n)\|^{-2\bar{\beta}(\alpha)})
	\]
	and for $J_3$ we have
	\begin{align*}
	|J_3|^{2}&\leq 4\lambda^{-2}\ \left(\EE_n\int\left|\frac{\lamhat}{\lambda}-1\right||\phat^{*}-p^{*}|p\,d\theta\right)^{2}\\
	&\leq 4\lambda^{-2}\ \EE_n\left[\left(\frac{\lamhat}{\lambda}-1\right)^{2}\|p\|_2^2\right]\cdot \MISE_n(p^{*},\phat^{*}).
	\end{align*}	
	Theorem \ref{thm::bound_unnorm_est}(c) and part (a) again imply
	\[
	|J_3|=O(M^{2-2\bar{\beta}(\alpha)}\|\N(n)\|^{-2\bar{\beta}(\alpha)}).
	\]
	Combining the above estimates and using \eqref{intermMISE1} proves part (b).
\end{proof}

\subsection{Proof of the error bound for the interpolated estimator}
	
	\begin{proposition}\label{essboundphatder}
		Suppose hypotheses \eqref{hyp0:model}-\eqref{hyp6:hmaxcond} hold. Then for $0\leq\pi<k-1$
		\[
		\left| \phat_{m}^{(\pi)}(\theta;\omega) \right|\leq C\|\N(n)\|^{\pi \alpha} \quad \theta \in (a,b), \quad \omega \in \underline{\Omega}^{M,\N(n)},
		\]
		where the constant $C=C(\pi,k,p_{m})$.
	\end{proposition}
	
	\begin{proof} Observe that the estimator $\phat_{m}$ can be expressed as
		\[
		\phat_{m}(\theta)=\exp{\big[ \sum_{j=0}^{J_{m}(n)-1}\yhat_{j}^{m}B_{j,k}(\theta)-c(\yhat^{m}) \big]}=\exp{\big[\sum_{j=0}^{J_{m}(n)-1}(\yhat_{j}^{m}-c(\yhat^{m}))B_{j,k}(\theta) \big]}
		\]
		
		\noindent Then, applying Faa di Bruno's formula, we obtain for $\theta\in [t_{i},t_{i+1}]$
		\[
		|\phat_{m}^{(\pi)}(\theta)|\leq \phat_{m}(\theta)\sum_{k_1+2k_2+\dots+\pi k_{\pi}=\pi}\frac{\pi !}{k_1!k_2!\dots k_{\pi}!}\prod_{i=1}^{\pi}\left(\frac{\left|\frac{d^{i}}{d\theta^{i}}\sum_{j=0}^{J_{m}(n)-1}(\yhat_{j}^{m}-c(\yhat^{m}))B_{j,k}(\theta)\right|}{i!}\right)^{k_i},
		\]
		where $k_1,\dots,k_{\pi}$ are nonnegative integers and if $k_{i}>0$ with $i\geq k$ then that term in the sum above will be zero since almost everywhere $B_{j,k}^{(i)}(\theta)=0$. By De Boor's formula \cite[p.132]{de Boor}, we can estimate the derivative of a spline as follows
		\[
		\left|\frac{d^{i}}{d\theta^{i}}\sum_{j=0}^{J_{m}(n)-1}(\yhat_{j}^{m}-c(\yhat^{m}))B_{j,k}(\theta)\right|=\left|\frac{d^{i}}{d\theta^{i}}\log{\phat_{m}(\theta)}\right|\leq C\frac{\|\log{\phat}\|_{\infty}}{\underline{h}_{m}^{i}}.
		\]
		where the constant $C$ depends only on the order $k$ of the B-splines. Therefore, we can bound $|\phat_{m}^{(\pi)}(\theta)|$ as follows
		\begin{align*}
		| \phat_{m}^{(\pi)}(\theta)|&\leq \phat_{m}(\theta)\sum_{k_1+2k_2+\dots+\pi k_{\pi}=\pi}\frac{\pi !}{k_1!k_2!\dots k_{\pi}!}\prod_{i=1}^{\pi}\left(C\frac{\|\log{\phat}_{m}\|_{\infty}}{i!\, \underline{h}_{m}^{i}}\right)^{k_i}\\
		&\leq \phat_{m}(\theta) \bigg( \frac{1+C^{\pi}\|\log{\phat}_{m}\|_{\infty}^{\pi}}{\underline{h}_{m}^{\pi}}\bigg)\sum_{k_1+2k_2+\dots+\pi k_{\pi}=\pi}\frac{\pi !}{k_1!k_2!\dots k_{\pi}!}.
		\end{align*}	
		The above leads to the following bound:
		\begin{align*}
		\left|\phat_{m}^{(\pi)}(\theta)\right|&\leq \phat_{m}(\theta)\frac{1+C^{\pi}\|\log{\phat}_{m}\|_{\infty}^{\pi}}{\underline{h}_{m}^{\pi}}\sum_{\zeta=1}^{\pi}\frac{\pi !}{\zeta !}(\pi-\zeta+1)^{\zeta}\\
		&\leq C(k,\pi)\,\phat_{m}(\theta)\frac{1+\|\log{\phat}_{m}\|_{\infty}^{\pi}}{\underline{h}_{m}^{\pi}}
		\end{align*}
		where $C(k,\pi)$ is a constant that depends on the order $k$ and the $\pi$.
		Next, recalling the hypotheses \eqref{hyp1:unifconvNh}, \eqref{hyp2:Lchoice},\eqref{hyp4:pcond1} and \eqref{hyp6:hmaxcond}, we obtain
		\[
		\phat_{m}(\theta)\leq |\phat_{m}(\theta)-p_{m}(\theta)|+p_{m}(\theta)\leq \|p_{m}\|_{\infty}(1+c\|\N(n)\|^{-\bar{\beta}(\alpha)})
		\]
		and
		\begin{align*}
		\|\log{\phat_{m}}\|_{\infty}&\leq \|\log{\phat_{m}}-\log{\pbar_{m}}\|_{\infty}+\|\log{\pbar_{m}}-\log{p_{m}}\|_{\infty}+\|\log{p_{m}}\|_{\infty}\\
		&\leq c\|\N(n)\|^{-\tilde{\beta}(\alpha)}+\|\log{p_{m}}\|_{\infty}
		\end{align*}
		where we also used Lemma \ref{pbarsupnormbound} and Theorem \ref{estonOmegan} and $\tilde{\beta}(\alpha)=\min(1/2-\alpha,\alpha q)$. Therefore,
		\begin{align*}
		\left|\phat_{m}^{(\pi)}(\theta)\right|&\leq C(k,\pi)\,\|p_{m}\|_{\infty}(1+\|\N(n)\|^{-\tilde{\beta}(\alpha)})\frac{1+\|\N(n)\|^{-\pi\beta}+\|\log{p_{m}}\|_{\infty}^{\pi}}{\underline{h}_{m}^{\pi}}\\
		&\leq C(\pi,k,p_{m})\frac{1}{\underline{h}_{m}^{\pi}}\\
		&= C(\pi,k,p_{m})\underline{h}_{m}^{-\pi} \ \sim \ C(\pi,k,p_{m})\|\N(n)\|^{\pi \alpha}
		\end{align*}
		The final result follows immediately and since the index $i$ was chosen arbitrarily and that all interior knots are simple, this concludes the proof.
	\end{proof}
	
	\begin{remark}
		Remark \ref{CurryScho}, in Section \ref{section:appendix}, allowed us to extend the bound for all $\theta \in (a,b)$ in the proof above. In reality, we can also extend the bound to the closed interval $[a,b]$. Since $a=t_{0}$ and $b=t_{K_{m}(n)}$ are knots with multiplicity $k$, any B-spline that isn't continuous at those knots will just be a polynomial that has been cut off, which means there is no blow-up. Thus, we can extend the bound by considering right-hand and left-hand limits of derivatives at $a$ and $b$, respectively. From this point on we consider the bound in Proposition \ref{essboundphatder} holds for all $\theta \in [a,b]$.
	\end{remark}

	\begin{lemma} \label{subpostboundlmm}
		Let hypotheses \eqref{hyp0:model}-\eqref{hyp6:hmaxcond} hold. Suppose that for each $m=1,\dots,M$ and every $\omega \in \underline{\Omega}^{M,\N(n)},$ the map $\theta \to \hat{p}_m(\theta; \omega) \in C^{(\pi)}([a,b])$ for some $0<\pi<k-1$.	Then
		\begin{equation*}
		\Big|\frac{d^{\pi}}{d \theta^{\pi}}\hat{p}^*(\theta;\omega)\Big| =
		\big|(\hat{p}_{1}...\hat{p}_{M})^{(\pi)}(\theta)\big| \leq C\|\N(n)\|^{\pi \alpha}M^{\pi}  \quad \theta \in [a,b], \quad \omega \in \underline{\Omega}^{M,\N(n)},
		\end{equation*}
		where $C=C(\pi,k,p_1,\dots,p_M)$.
	\end{lemma}
	\begin{proof}
		Let $\theta\in [a,b]$. By Proposition \eqref{essboundphatder} we have
		\[
		|\hat{p}_m^{(\pi)}(\theta)| \, \leq \, C(\pi,k,p_{m})\|\N(n)\|^{\pi \alpha}.
		\]
		Then, using the general Leibnitz rule and employing the above inequality we obtain
		\begin{align*}
		\Big|\frac{d^{\pi}}{d \theta^{\pi}}\hat{p}^*(\theta)\Big| &=
		\big|(\hat{p}_{1}...\hat{p}_{M})^{(\pi)}(\theta)\big| = \\
		&  = \bigg|\sum_{i_{1}+\dots+i_{M}=\pi}\dfrac{\pi!}{i_{1}! \dots i_{M}!}\hat{p}_{1}^{(i_{1})}...\hat{p}_{M}^{(i_{M})}\bigg|
		\\
		& \leq \sum_{i_{1}+...+i_{M}=\pi}\dfrac{\pi!}{i_{1}!...i_{M}!}C(i_{1},k,p_{1})\|\N(n)\|^{i_{1} \alpha}\,...\,C(i_{M},k,p_{M})\|\N(n)\|^{i_{M} \alpha}\\
		&=\|\N(n)\|^{\pi \alpha}\sum_{i_{1}+...+i_{M}=\pi}\dfrac{\pi!}{i_{1}!...i_{M}!}C(i_{1},k,p_{1})\,...\,C(i_{M},k,p_{M})
		\end{align*}
		From the proof of Proposition \ref{essboundphatder}, notice that $C(i,k,p_{m})\leq C(j,k,p_{m})$ for positive integers $i\leq j$. Therefore, we have
		\[
		|\hat{p}_m^{(\pi)}(\theta)|\leq C(\pi,k,p_{1},\dots,p_{M})\|\N(n)\|^{\pi \alpha}\sum_{i_{1}+...+i_{M}=\pi}\dfrac{\pi!}{i_{1}!...i_{M}!}
		\]
		where $C(\pi,k,p_{1},\dots,p_{M})=C(\pi,k,p_{1})\dots C(\pi,k,p_{M})$ and the result follows from the multinomial theorem. This concludes the proof.	
	\end{proof}

\noindent{\bf Proof of Theorem \ref{thm::bound_interp_est}}	
\begin{proof}
		Let  $i \in \{0,\dots,N-1\}$. By Lemma \ref{estinterrlmm}, Lemma \ref{subpostboundlmm}, and
		\eqref{compintp}   for any $\theta \in [x_{il},x_{(i+1)l}]$ we have
		\begin{equation}\label{locerr}
		\begin{aligned}
		\big|\phat^{*}(\theta) - \ptil^{*}(\theta)\big|  &= \big|\phat^{*}(\theta) - \hat{q}_i(\theta)\big| \\[2pt]
		& \leq \bigg( \sup_{\theta \in [x_{il},x_{(i+1)l}]} \Big|\frac{d}{d \theta}^{(l+1)}\hat{p}^*(\theta)\Big|\bigg) \dfrac{(\Delta x)^{l+1}}{4(l+1)} \\
		& \leq C_1 \frac{(\Delta x)^{l+1}}{4(l+1)} \|\N(n)\|^{(l+1)\alpha}M^{l+1}.
		\end{aligned}
		\end{equation}
		where $C_1 = C(l+1,k,p_{1},\dots,p_{M})$. Thus 
		\[
		\begin{aligned}
		\mathbb{E}\int_{a}^{b}\big(\phat^{*}(\theta)-\ptil^{*}(\theta)\big)^{2}d\theta & = \sum_{i=0}^{N-1} \mathbb{E}\int_{x_{il}}^{x_{(i+1)l}} \big( \phat^{*}(\theta)-\hat{q}_i(\theta)\big)^2 d\theta \\
		&   \leq C_2 \Bigg( \frac{(\Delta x)^{l+1}}{4(l+1)}\|\N(n)\|^{(l+1) \alpha} M^{l+1}  \Bigg)^2
		\end{aligned}
		\]
		where $C_2=C_1^{2}(b-a)$. This proves part (a).

		We next observe that
		\[
		|\lamhat-\tilde{\lambda}|\leq \int_{a}^{b}|\phat^{*}(\theta)-\ptil^{*}(\theta)|\, d\theta
		\]
		and hence \eqref{locerr} and  Lemma \ref{estinterrlmm} imply part (b).

	To analyze $\MISE(\phat,\tilde{p})$ we introduce an intermediate functional
	\begin{equation*}
		\underline{\MISE}\left(\phat,\ptil \ | \ \underline{\Omega}^{M,\N(n)}\right)=\EE_{\underline{\Omega}^{M,\N(n)}}\left[ \left(\frac{\tilde{\lambda}}{\lamhat(\omega)}\right)^{2}\int (\phat(\theta;\omega)-\ptil(\theta))^{2}\,d\theta \right].
		\end{equation*}
We now show that functional is asymptotically equivalent to $\MISE$, thus providing us with the means to view how $\MISE$ scales without having to directly analyze it. Notice that 
\begin{align*}
&\underline{\MISE}\left(\phat,\ptil \ | \ \underline{\Omega}^{M,\N(n)}\right)=\EE_{\underline{\Omega}^{M,\N(n)}}\left[ \left(\frac{\tilde{\lambda}}{\lamhat}-1+1\right)^{2}\int (\phat(\theta;\omega)-\ptil(\theta))^{2}\,d\theta \right]\\
&=\EE_{\underline{\Omega}^{M,\N(n)}}\left[\left( \lambda^{-2}\left(\frac{\lambda}{\lamhat}\right)^{2}\left(\tilde{\lambda}-\lamhat\right)^{2}+2\lambda^{-1}\frac{\lambda}{\lamhat}\left(\tilde{\lambda}-\lamhat\right)+1\right)\int (\phat(\theta;\omega)-\ptil(\theta))^{2}\,d\theta \right]\,.
\end{align*}	
Then, Theorem \ref{MISEpphat}(a) implies
\begin{align*}
\frac{\lambda}{\lamhat}&\leq \frac{1}{1-C\,M^{1-\bar{\beta}(\alpha)}\|\N(n)\|^{-\bar{\beta}(\alpha)}}\,,
\end{align*}
and hence, by part (b), for large enough $n$ for which $1-C\,M^{1-\bar{\beta}(\alpha)}\|\N(n)\|^{-\bar{\beta}(\alpha)}>0$, we obtain
\begin{equation*}
\underline{\MISE}\left(\phat,\ptil \ | \ \underline{\Omega}^{M,\N(n)}\right)=(1+O(M^{l+1}(\Delta x)^{l+1}))\MISE\left(\phat,\ptil \ | \ \underline{\Omega}^{M,\N(n)}\right).
\end{equation*}

Let $\EE_n(\cdot):=\EE(\cdot|\underline{\Omega}^{M,\N(n)})$. Notice that
\begin{align*}
&\underline{\MISE}\left( \phat,\ptil \ | \ \underline{\Omega}^{M,\N(n)} \right)
=\EE_n\int \left(\frac{\tilde{\lambda}}{\lamhat}\phat-\frac{1}{\lamhat}\ptil^{*}-\phat+\phat\right)^{2}\,d\theta\\
&\leq \frac{\lambda^{-1}}{1-C\,M^{1-\bar{\beta}(\alpha)}\|\N(n)\|^{-\bar{\beta}(\alpha)}}\EE_n\int \left((\tilde{\lambda}-\lamhat)(\phat-p)+(\tilde{\lambda}-\lamhat)p+(\phat^{*}-\ptil^{*})\right)^{2}\,d\theta\\
&\leq \frac{\lambda^{-1}}{1-C\,M^{1-\bar{\beta}(\alpha)}\|\N(n)\|^{-\bar{\beta}(\alpha)}}(J_1+J_2+J_3+J_4+J_5+J_6)
\end{align*}
where
\begin{align*}
&J_1=\EE_n\int (\tilde{\lambda}-\lamhat)^{2}(\phat-p)^{2}\,d\theta,&
&J_2=\EE_n\int (\tilde{\lambda}-\lamhat)^{2}p^{2}\,d\theta,&\\
&J_3=\EE_n\int (\phat^{*}-\ptil^{*})^{2}\,d\theta,&
&J_4=2\,\EE_n\int (\tilde{\lambda}-\lamhat)^{2}(\phat-p)p\,d\theta,&\\
&J_5=2\,\EE_n\int (\tilde{\lambda}-\lamhat)(\phat-p)(\phat^{*}-\ptil^{*})\,d\theta,&
&J_6=2\,\EE_n\int (\tilde{\lambda}-\lamhat)(\phat^{*}-\ptil^{*})p\,d\theta.&
\end{align*}

Then hypotheses \eqref{hyp1:unifconvNh}-\eqref{hyp6:hmaxcond}, Lemma \ref{MISEpphat}, part (a) and part (b) imply
	\begin{align*}
	\vert J_1\vert &\leq C_1\Big( \|\N(n)\|^{\alpha} (\Delta x) M  \Big)^{2(l+1)} M^{2-2\bar{\beta}(\alpha)}\|\N(n)\|^{-2\bar{\beta}(\alpha)}\\
	\vert J_2\vert +\vert J_3\vert +\vert J_6\vert &\leq C_2\Big( \|\N(n)\|^{\alpha} (\Delta x) M  \Big)^{2(l+1)}\\
	\vert J_4\vert +\vert J_5\vert &\leq C_3\Big( \|\N(n)\|^{\alpha} (\Delta x) M  \Big)^{2(l+1)}  M^{1-\bar{\beta}(\alpha)}\|\N(n)\|^{-\bar{\beta}(\alpha)}
	\end{align*}
	which for large $n$ implies part (c).
\end{proof}

\noindent{\bf Proof of Theorem \ref{MISEestthm}}

	\begin{proof}
		Observe that
		\begin{equation*}
		\begin{aligned}
		\MISE (p\,,\ptil \ | \ \underline{\Omega}^{M,\N(n)}) 
		&  \leq  \mathbb{E}\int_{a}^{b}\big(p(\theta)-\hat{p}(\theta)\big)^{2}d\theta+\mathbb{E}\int_{a}^{b}\big(\hat{p}(\theta)-\ptil(\theta)\big)^{2}d\theta
		\\ &  =: I_1 + I_2.
		\end{aligned}
		\end{equation*}
		Then the estimate \eqref{MISEest} follows from Theorem \ref{MISEpphat}(b) and Theorem \ref{thm::bound_interp_est}(c). Using that estimate we can ask the following question. Suppose that we chose $\Delta x$ to be a
		function of the number of samples so that
		\begin{equation}\label{dxfunc}
		c_{1}\|\N(n)\|^{-\kappa} \leq \Delta x (n) \leq c_{2} \| \N(n) \|^{-\kappa}
		\end{equation}
		for some constants $c_{1},c_{2}$ and $\kappa$. Clearly, one would not like $\Delta x$ to be excessively small in order to
		avoid difficulties that appear with round-off error when computing. On the other hand one would like the error
		to converge to zero as fast as possible. To find the smallest rate $\kappa$ for which the asymptotic
		rate achieves its maximum, we define the function
		\[
		R(\kappa) := -\lim_{\|\N(n)\| \to \infty} \log_{\|\N(n)\|} \MISE(p,\ptil\ | \ \underline{\Omega}^{M,\N(n)})
		\]
		that describes the asymptotic rate of convergence of the mean integrated squared error. By \eqref{MISEest} we have
		\begin{equation}\label{miserate}
		R(\kappa) = \left\{
		\begin{aligned}
		& 2\bar{\beta}(\alpha),& \kappa  \geq  \bar{\beta}(\alpha)\left(\frac{1}{l+1}+\frac{1}{q}\right) & \\
		& \left(\kappa-\alpha\right)2(l+1), & \kappa < \bar{\beta}(\alpha)\left(\frac{1}{l+1}+\frac{1}{q}\right) &
		\end{aligned}\right.
		\end{equation}
		It is obvious that the smallest rate for which the function $R(\kappa)$ achieves its maximum value of
		$2\bar{\beta}(\alpha)$ is given by $\kappa=\bar{\beta}(\alpha)\left(\frac{1}{l+1}+\frac{1}{q}\right)$. Then \eqref{dxfunc} and \eqref{miserate} imply \eqref{MISEestthm_b}.
	\end{proof}

		\section{Examples}\label{section:num_exp}
	
	\begin{figure}
		\centering
		\begin{subfigure}[t]{0.45\textwidth}
			\centering
			\includegraphics[width=\textwidth]{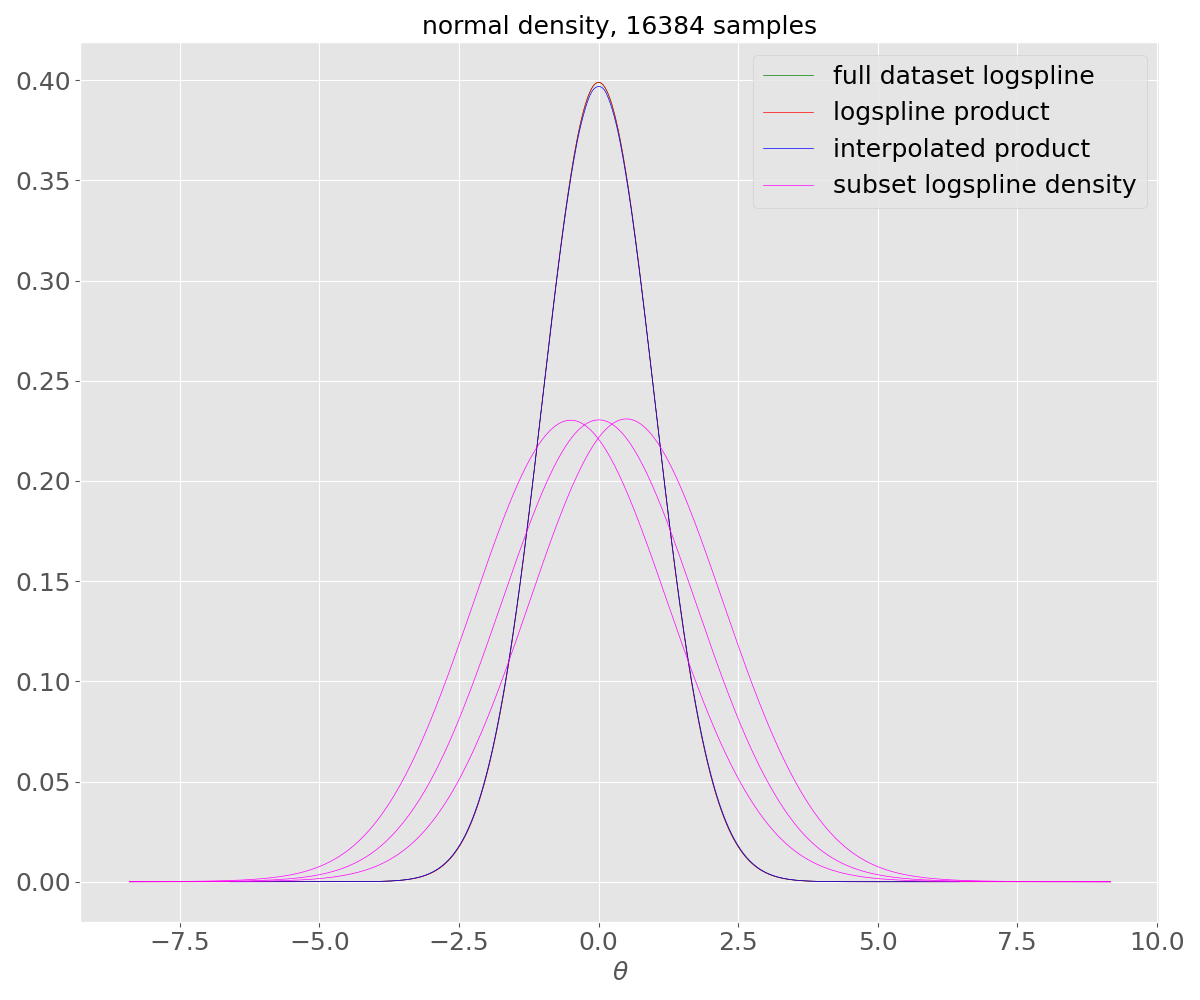}
			\caption{\footnotesize Subset and full-data posterior densities}\label{fig::norm_dens}
		\end{subfigure}
		~~
		\begin{subfigure}[t]{0.45\textwidth}
			\centering
			\includegraphics[width=\textwidth]{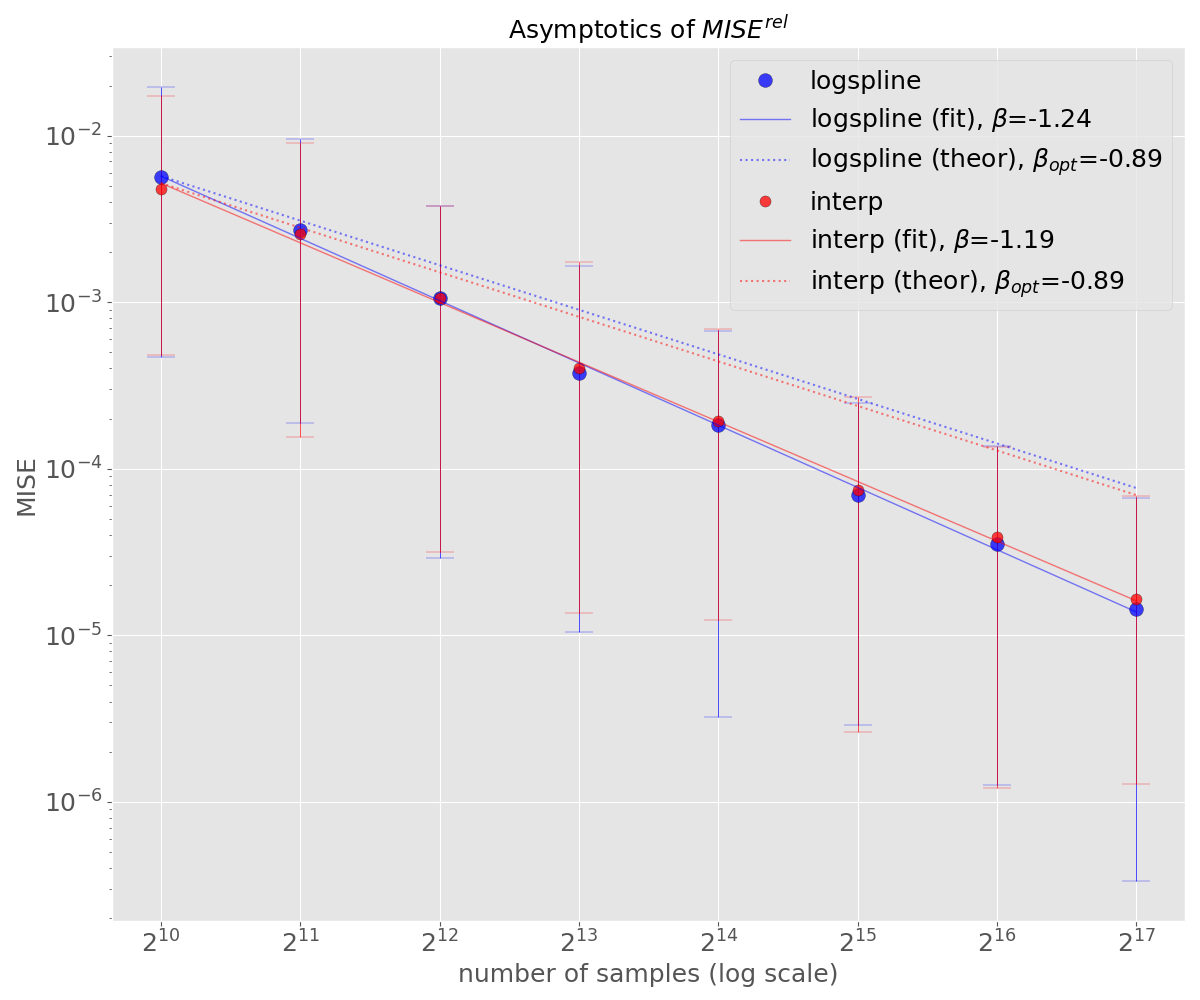}
			\caption{\footnotesize Error rate comparison}\label{fig::norm_error}
		\end{subfigure}
		\caption{\footnotesize The numerical result for the normal sub-posterior densities and their corresponding error rates.}\label{fig::norm_exp}
	\end{figure}
	
	\subsection{Numerical experiment with normal subset posterior densities}
	In this section we assume that the dataset $\textbf{x} = \{x_1,\dots,x_n\}$ contains random samples from $\mathcal{N}(\mu,\sigma^2)$ with unknown $\mu$ and known $\sigma^2$ and we will be constructing the posterior density for $\mu$. Consider the estimator introduced in this manuscript in the form
	\begin{equation}\label{Ex.1::Neis}
	p(\theta|\textbf{x})\propto p(\theta)p(\textbf{x}_1,\dots,\textbf{x}_M|\theta)\propto\prod_{m=1}^{M}p(\theta|\textbf{x}_m), \quad \text{where} \quad p(\theta|\textbf{x}_m)=p(\textbf{x}_m | \theta)p(\theta)^{1/M}.
	\end{equation}
	For each sub-posterior density we  choose the prior to be $p(\theta)^{1/M}$ where $p(\theta)$ is an uninformative prior. In this case each subset posterior has the form
	\[
	p(\theta|\textbf{x}_{m}) \propto L(\theta|\textbf{x}_{m}) \sim \mathcal{N}\left(\frac{T_m}{n_m},\frac{\sigma^2}{n_m}\right), \quad T_m = \sum_{x\in \textbf{x}_{m}} x
	\]
	where each $n_m=|\textbf{x}_{m}|$. Plugging in the above form of each subset posterior into \eqref{Ex.1::Neis}, a detailed calculation yields the following result for the form of the full-data posterior
	\[
	p(\theta|\textbf{x})\sim \mathcal{N}\left(\frac{T}{n},\frac{\sigma^2}{n}\right),\quad T=\sum_{m}T_m,
	\]
	and $n=\sum_{m}n_m$ is the total number of samples.
	
	Since the main objective of our paper is the theoretical derivation of the convergence rate of the Mean Intergrated Squared Error outlined in Theorem \ref{MISEestthm}, we provide the numerical evidence for this theoretical rate in our example by partitioning the samples into $M=3$ subsets and $n_m=\frac{n}{3}$. Then, we choose $\frac{T_m}{n_m} = (m-2)\frac{1}{2}, \ m=1,2,3$, and $\frac{\sigma^2}{n_m}=3$. This implies that $T=0$ and $\frac{\sigma^2}{n}=1$. This leads to the subset posteriors being
	\[
	p(\theta|\textbf{x}_m) \sim \mathcal{N}\left(\frac{m-2}{2},3\right)
	\]
	which in turn yields the full-data posterior of $\theta=\mu$ to be
	\[
	p(\theta|\textbf{x}) \sim \mathcal{N}\left(0,1\right).
	\]
Figure \ref{fig::norm_exp} (a) shows the expected subset posterior densities as well as the full posterior density estimators. Figure \ref{fig::norm_exp} (b) shows the convergence rates.

	\subsection{Numerical experiment with gamma subset posterior densities}
	
	\begin{figure}
		\centering
		\begin{subfigure}[t]{0.45\textwidth}
			\centering
			\includegraphics[width=\textwidth]{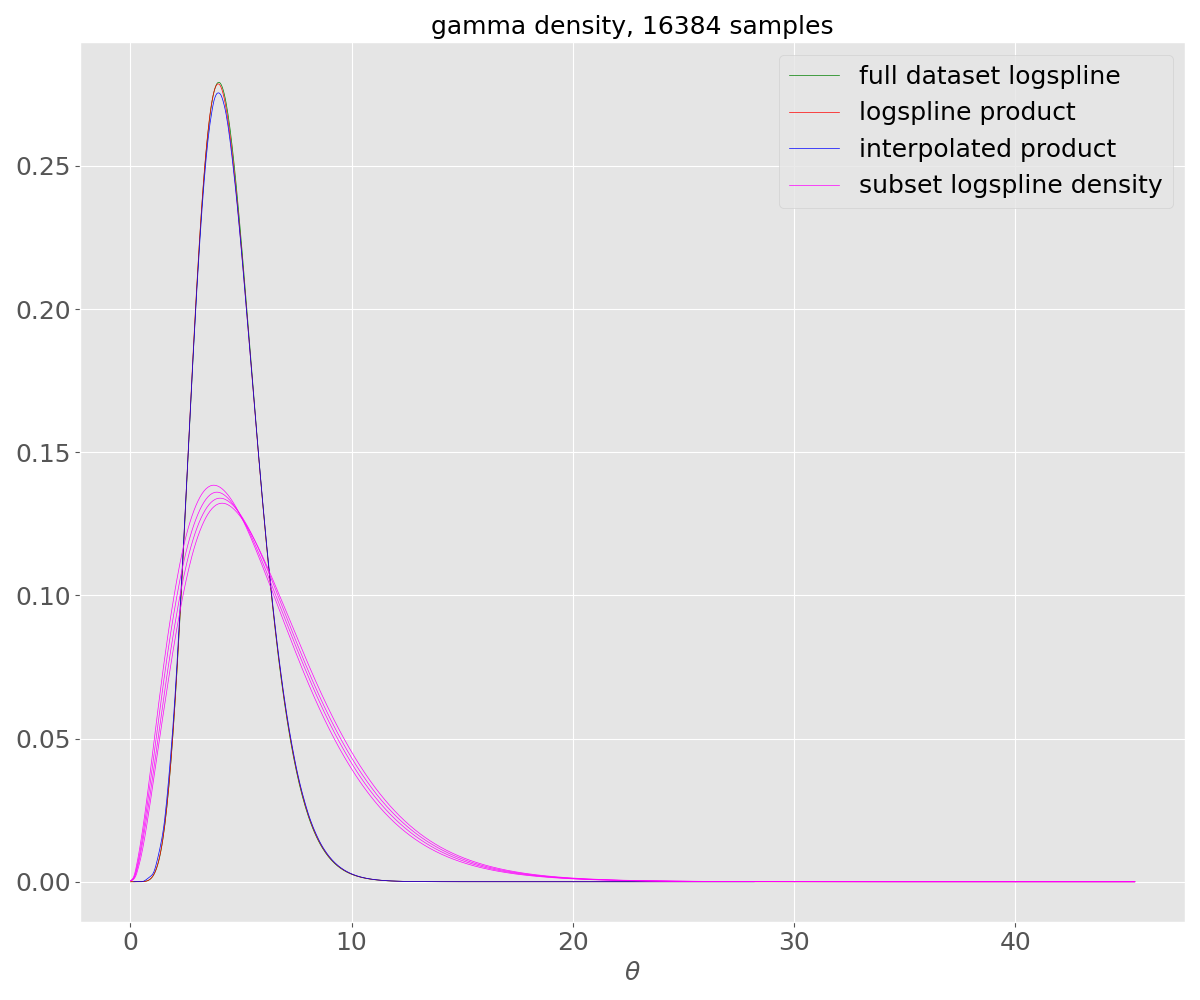}
			\caption{\footnotesize Subset and full-data posterior densities}\label{fig::gamma_dens}
		\end{subfigure}
		~~
		\begin{subfigure}[t]{0.45\textwidth}
			\centering
			\includegraphics[width=\textwidth]{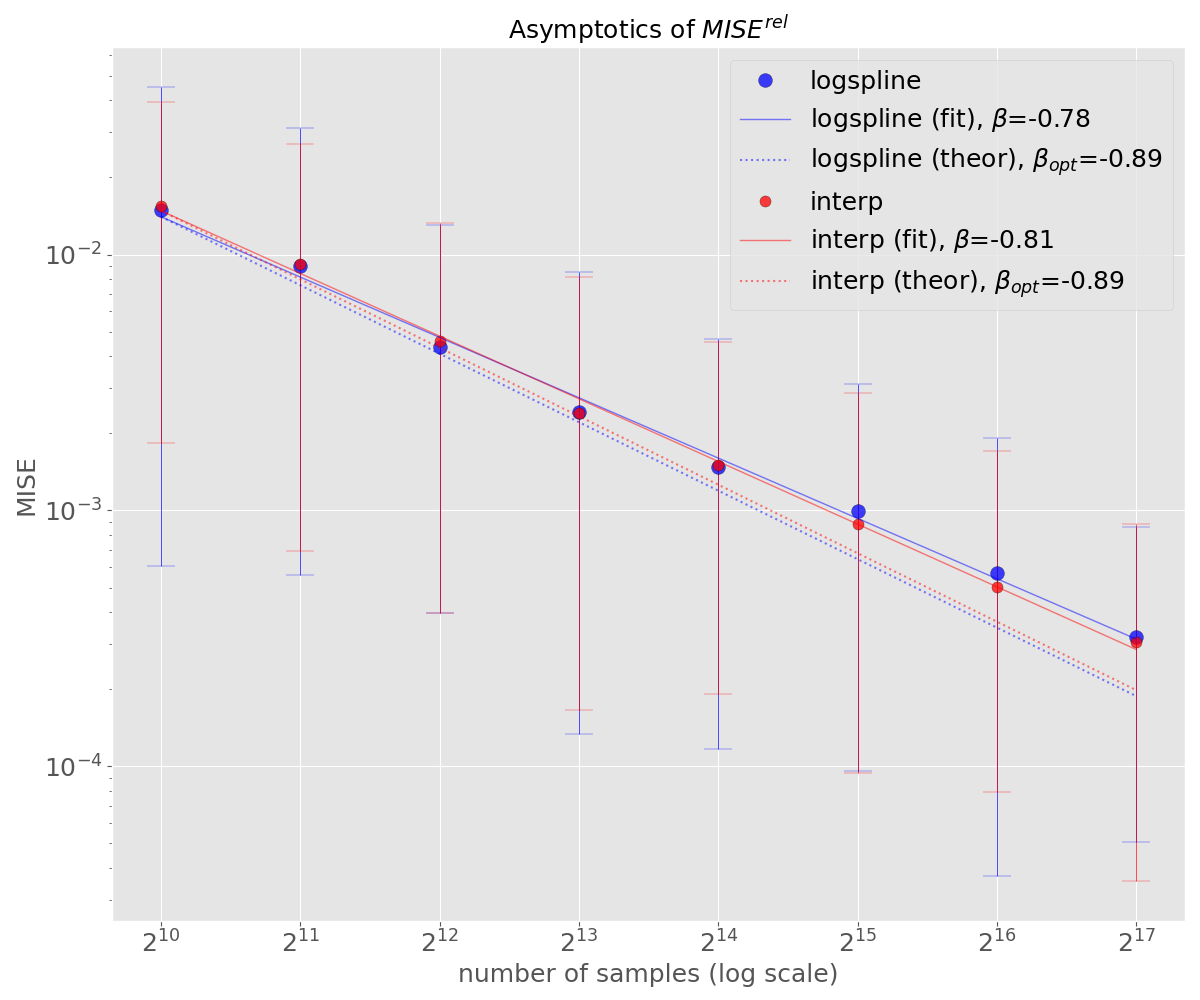}
			\caption{\footnotesize Error rate comparison}\label{fig::gamma_error}
		\end{subfigure}
		\caption{\footnotesize The numerical result for the gamma sub-posterior densities and their corresponding error rates.}\label{fig::gamma_exp}
	\end{figure}

	In this section we assume that the dataset $\textbf{x} = \{x_1,\dots,x_n\}$ contains random samples from $Poisson(\theta)$ (with $\theta$ unknown), and we will be constructing the posterior density for $\theta$. For each subset posterior density we choose the prior
	\[
	p_m(\theta)= Gamma(\bar{\alpha},\bar{\beta})=:q(\theta),
	\]
	which gives the subset posterior in the form of
	\[
		p(\theta|\textbf{x}_m) \propto L(\theta|\textbf{x}_m)p_m(\theta)= Gamma\left(\bar{\alpha}+T_m,n_m + \bar{\beta}\right), \quad T_m = \sum_{y\in \textbf{x}_m}y.
	\]
	We then take the full dataset prior
	\[
	p(\theta)=Gamma(M(\bar{\alpha}-1)+1, M\bar{\beta}) \propto q(\theta)^M
	\]
	which leads to
	\begin{equation*}
	p(\theta|\textbf{x}_1,\dots,\textbf{x}_M)=Gamma\left(T + M(\bar{\alpha}-1)+1, n + M\bar{\beta}\right), \quad T=\sum_{m=1}^M T_m,\ n=\sum_{m=1}^M n_m.
	\end{equation*}
	From the above formula we also obtain
	\begin{equation*}
		p(\theta|\textbf{x}_1,\dots,\textbf{x}_M) \propto p(\theta)\prod_{m=1}^Mp(\textbf{x}_m|\theta) \propto \prod_{m=1}^M p(\theta|\textbf{x}_m), \quad \text{where} \ p(\theta|\textbf{x}_m)=p(\textbf{x}_m|\theta)p(\theta)^{1/M},
	\end{equation*}
	which is the full dataset posterior as proposed by Neiswanger et al. \cite{Neiswanger}.
	
	Since the main objective of our paper is the theoretical derivation of the convergence rate of the Mean Intergrated Squared Error outlined in Theorem \ref{MISEestthm}, we provide the numerical evidence for this theoretical rate in our example by partitioning the samples into $M=4$ subsets and $n_m=\frac{n}{4}$. Then,  we choose $T_m =1+\epsilon_m$, where $\epsilon_m = -\frac{1}{10}+(m-1)\frac{1}{15}, \ m=1,2,3,4$. This implies that $T=4$. We next pick $\bar{\alpha} = \frac{7}{4}$ and $\frac{n}{4}+\bar{\beta} = \frac{1}{2}$. This gives $n+M\bar{\beta} = 2$. This leads to the subset posteriors being
	\[
	p(\theta|\textbf{x}_m) = Gamma\left( 1+\epsilon_m, \frac{1}{2} \right)
	\]
	which in turn yields the full-data posterior of $\theta$ to be
	\[
	p(\theta|\textbf{x}) = Gamma(8,2).
	\]
 Figure \ref{fig::gamma_exp} (a) shows the expected subset posterior densities as well as the full posterior density estimators. Figure \ref{fig::gamma_exp} (b) shows the convergence rates. It is clear that in this example we wanted to test how skewness impacts the error rate. For this reason we chose $n$ appropriately because when $n$ is large the posterior is asymptotically normal.

\begin{figure}
		\centering
		\begin{subfigure}[t]{0.45\textwidth}
			\centering
			\includegraphics[width=\textwidth]{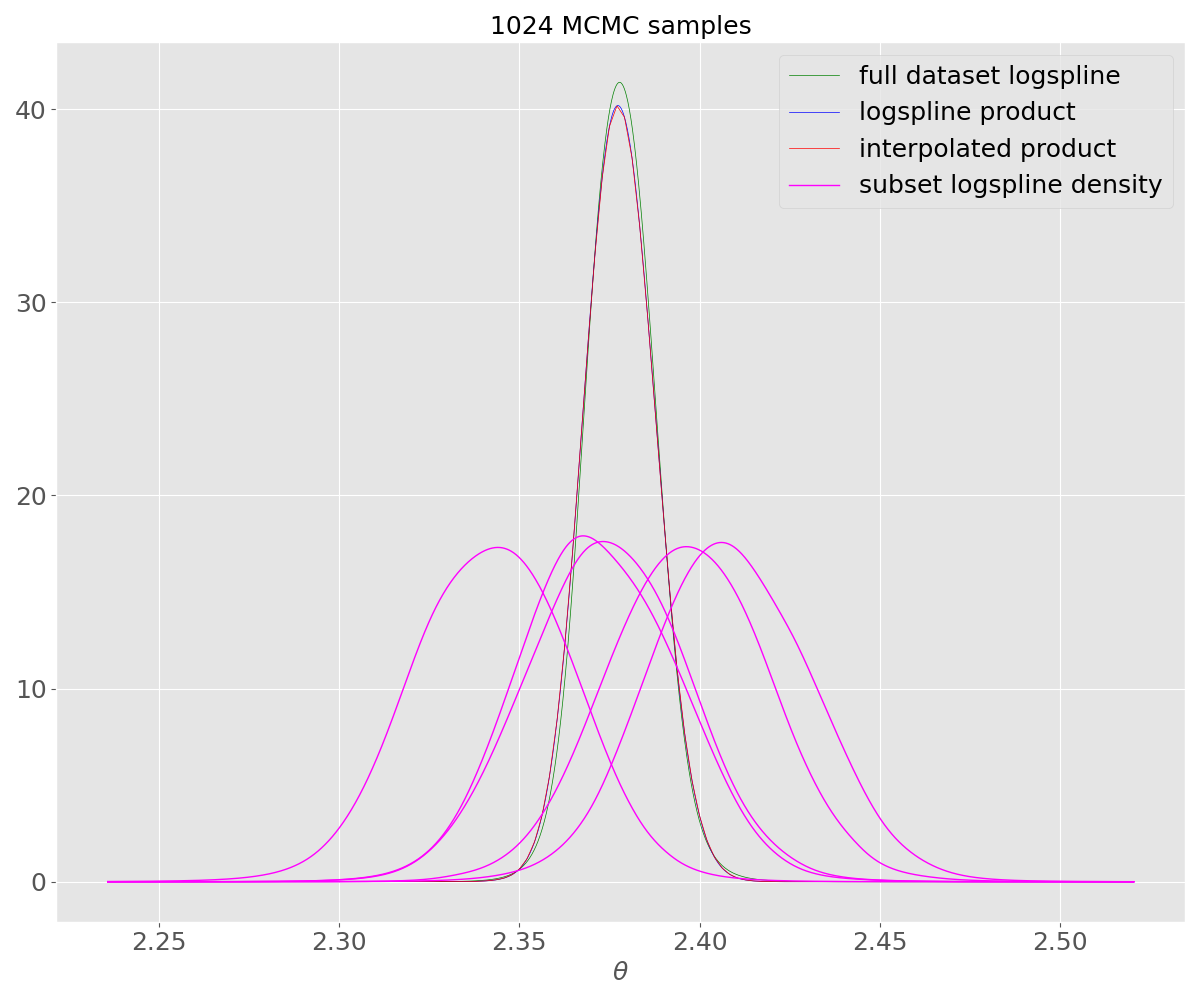}
		\end{subfigure}
		~~
		\begin{subfigure}[t]{0.45\textwidth}
			\centering
			\includegraphics[width=\textwidth]{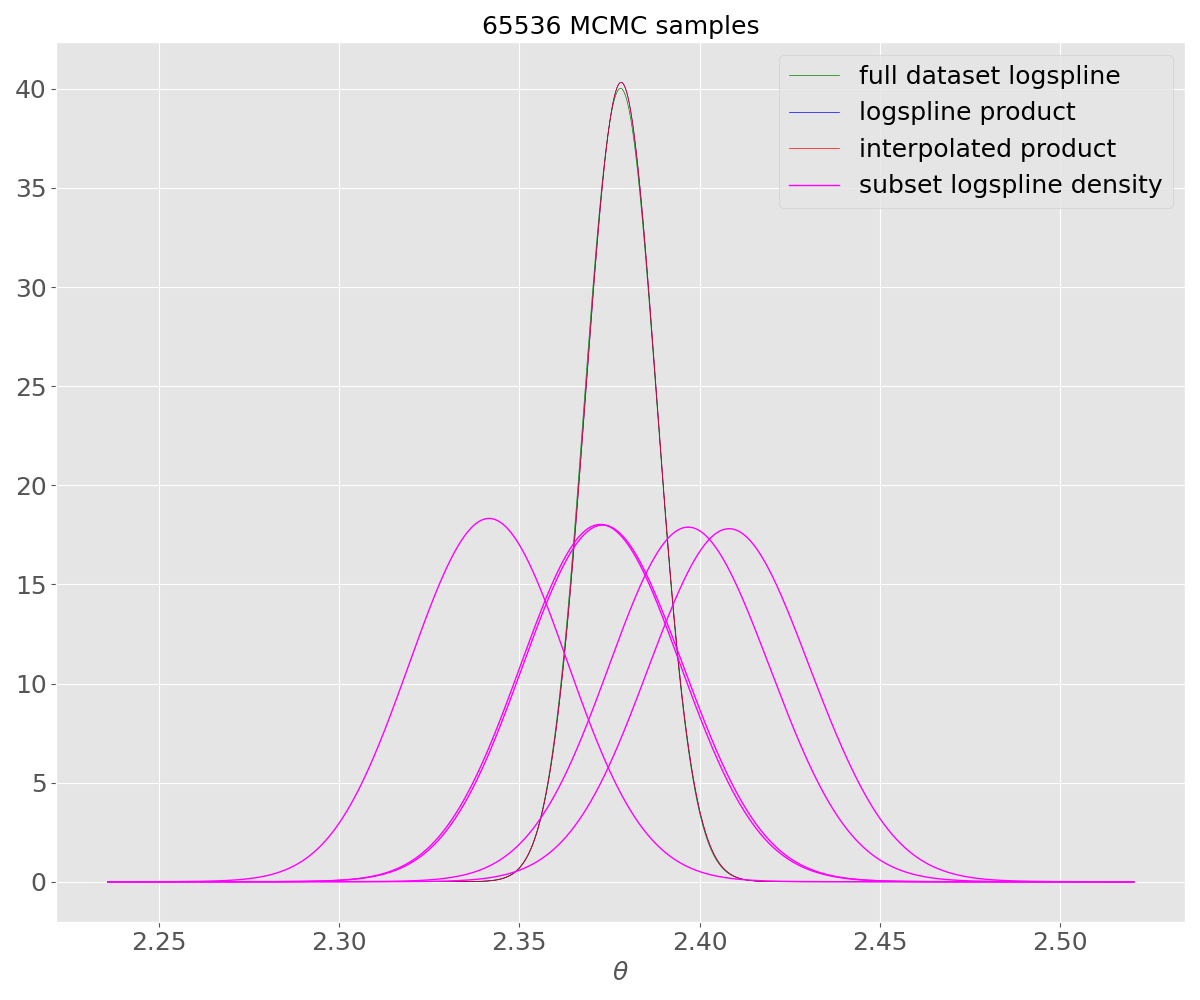}
		\end{subfigure}
		\caption{\footnotesize Posterior densities of the shape parameter $\alpha$ for flight arrival delay. $M=5$.}\label{fig::flight_dens}
\end{figure}

\subsection{Numerical experiment using real data on air flights}

In this experiment, we analyze real data obtained from the U.S. Department of Transportation \cite{DOT} on all commercial flights within the United States for the three-month period between January 2019 and March 2019.

The random variable of interest is the arrival delay in minutes for each flight, defined as the difference between the scheduled and actual arrival time. We removed values that are less than fifteen minutes, which are considered on-time \cite{DOT}. Doing so, we obtain a dataset containing  $n=101560$ samples, and we applied a square root transformation to the data set so that it follows an approximate Gamma distribution. The Bayesian model for the transformed data values $\{x_i\}_{i=1}^n$ is then chosen to be 
\[
x_i \sim \alpha,\beta, \quad  i=1,\dots,n.
\]

We estimate the shape parameter $\alpha$ and rate parameter $\beta$ using a Bayesian Gamma model with uninformative prior distributions for $\alpha$ and $\beta$ given by
\[
\alpha,\beta \sim Unif[0.00001,10000]
\]

We divide the full data set into M=5 subsets of size $n_j=20312$,  $j\in\{1,\dots,5\}$, and use the above uninformative priors for both full data set posterior as well as subset posteriors.

Using the python package \texttt{theano}, we generate samples of the pair $(\alpha,\beta)$ for each of the data subsets as well as the full dataset. We then use MCMC samples to produce estimates of marginal posterior densities. In our analysis, we focus on the shape parameter $\alpha$. In particular, we produce the full dataset logspline posterior density $\hat{p}(\alpha;N_0|X)$, with $N_0$  being the number of MCMC samples, the subset logspline posterior densities $\hat{p}(\alpha;N_j|X_j )$, $j=1,\dots,M$, and the estimate of the full dataset of the logspline posterior density $\phat(\alpha;N|X_1,\dots,X_m)$ constructed via the product of the subset posterior densities, where $N=(N_1,\dots,N_M)$, and $N_j$ the number of MCMC samples used to estimate the subset posterior density.

Since the true posterior is unknown, we estimate the expected squared difference between the full dataset posterior density of the shape parameter $\alpha$ and the one obtained by producting the subset posterior densities. In particular, we numerically estimate the following

\begin{equation}\label{empmise}
\MISE[\hat{p}(\alpha;N_0|X),\hat{p}(\alpha;N|X_1,\dots,X_m)]=
\EE\bigg[\int (\hat{p}(\alpha;N_0|X)-\hat{p}(\alpha;N|X_1,\dots,X_m))^2 d\alpha\bigg]
\end{equation}

\begin{figure}
		\centering
		\begin{subfigure}[t]{0.45\textwidth}
			\centering
			\includegraphics[width=\textwidth]{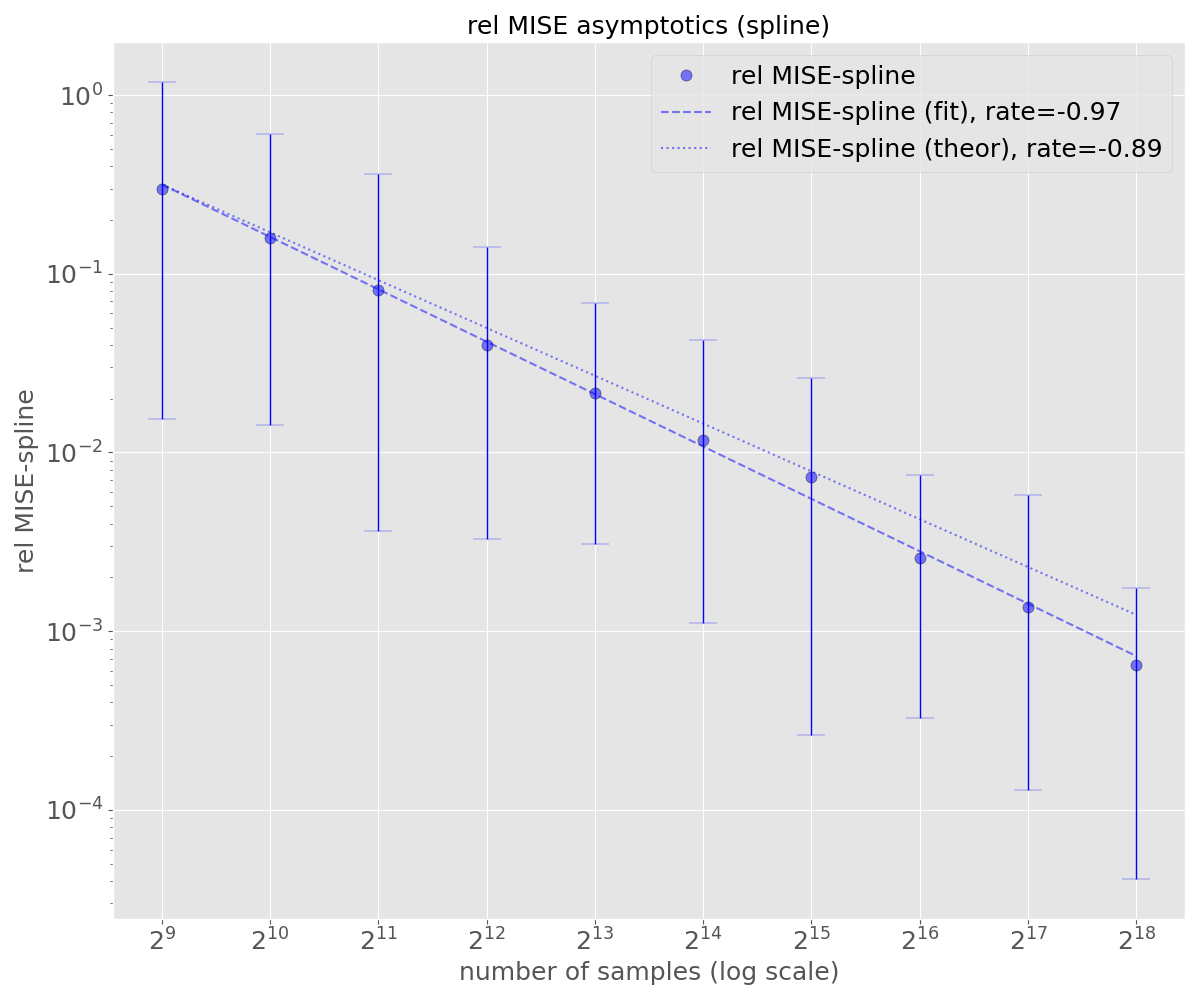}
		\end{subfigure}
		~~
		\begin{subfigure}[t]{0.45\textwidth}
			\centering
			\includegraphics[width=\textwidth]{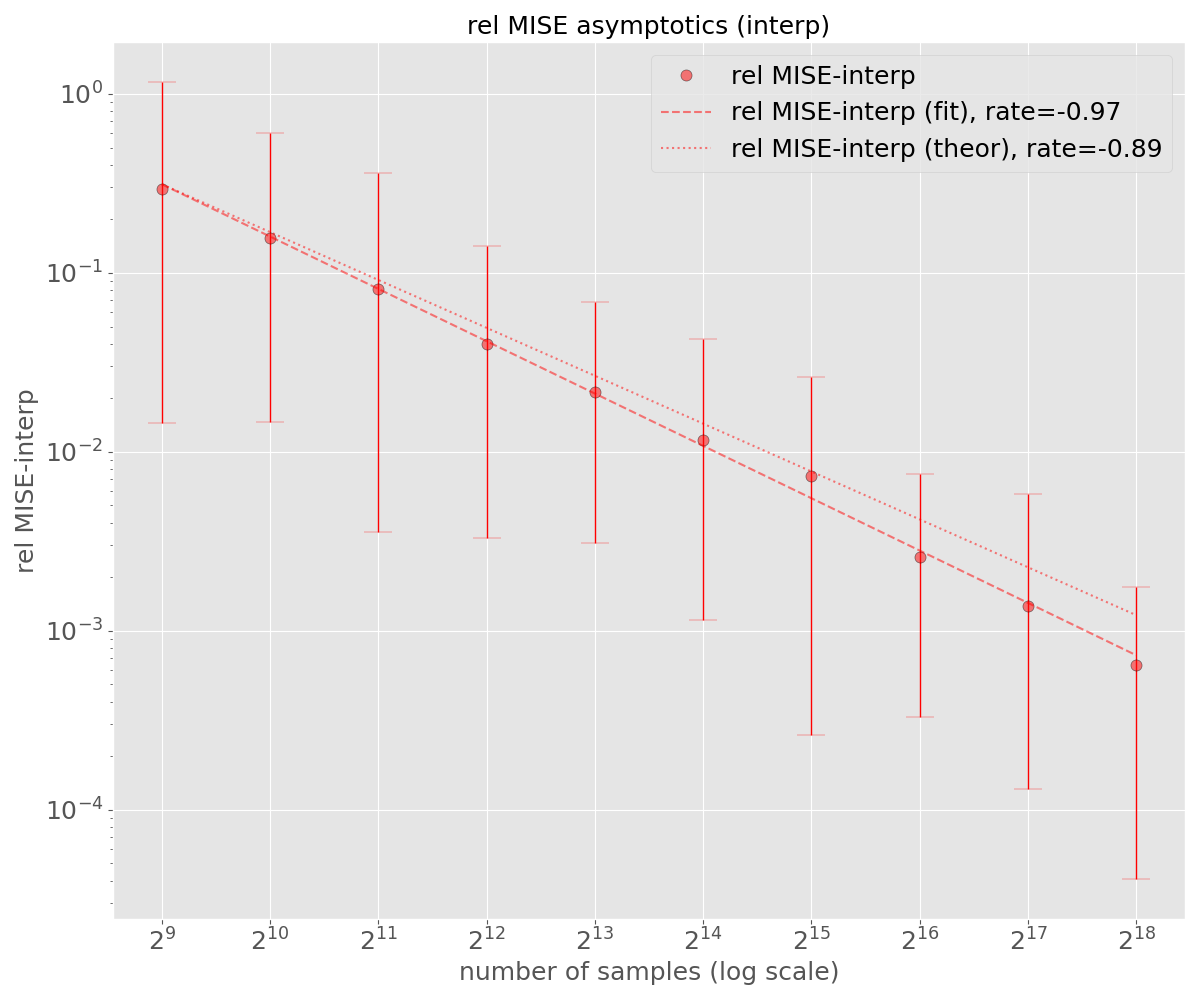}
		\end{subfigure}
		\caption{\footnotesize Convergence rates for product posterior density towards full dataset posterior.}\label{fig::flight_err}
\end{figure}

\subsubsection{Description of the experiment}

We generate $n=2^{18}$ MCMC samples of the pair $(\alpha,\beta)$ for the full dataset posterior density and the same number of samples for the posterior density on each of 5 data subsets. We use $4$   MCMC chains and $2^{16}$  iterations as burn-in. We then construct the estimated full dataset posterior density of $\alpha$ and its subset posterior density estimators using the logspline density estimation R package \texttt{logspline} based on the number of MCMC samples $n \in \{2^9,2^{10},\dots,2^{18}\}$. We then construct the product posterior estimator, which is compared to the estimated full set posterior by computing the integrated square difference
\begin{equation}\label{empise}
\int \Big(\hat{p}(\alpha;n|X)-\hat{p}(\alpha;N=(n,\dots,n)|X_1,\dots,X_m) \Big)^2 d\alpha.
\end{equation}

To estimate MISE defined in \eqref{empmise}, we repeat the above experiment $50$ times and then average the quantity in \eqref{empise}. We use these experiments to also construct the confidence interval for MISE for each $n\in\{2^9,2^{10},\dots,2^{18}\}$.

The full dataset, subset and product posteriors from the experiment are depicted in Figure~\ref{fig::flight_dens}. Observe that the variance of the subset posterior densities is larger than that of the full dataset posterior due to the fewer number of samples in each subset.

The rate of convergence of relative $\MISE(n)$ as a function of $n$ is depicted in Figure \ref{fig::flight_err} in a log-log scale and compared to the theoretical rate $2 q/(2q+1)=8/9$, where $q=4$ is the order of splines used in the package logspline. A similar analysis is performed for the interpolated version of logspline densities; see Figure \ref{fig::flight_err}. Observe that the numerical rate is indeed larger than the theoretical rate, which is consistent with the theoretical upper bound we obtained in Theorem \ref{thm::bound_unnorm_est}. We do believe that the reason why we are doing better than what our theory specifies is due to the fact that the numerical estimates obtained from the logspline density estimation package are based on maximizing the likelihood, which differs from our analysis based on the bias-variance trade-off. These numerics also indicate that the bias-variance analysis does not provide a tight estimate.

Figure \ref{fig::flight_knots} depicts the optimal number of knots selected by the logspline package based on the MLE. Observe that, to achieve our theoretical upper bound $n^{-q/(2q+1)}$  for MISE, the number of knots, assumed to be approximately uniformly distributed by hypothesis \eqref{hyp6:hmaxcond}, have to be picked according to the rate $1/(2q+1)=1/9$, where $q=4$ the order of splines used in the package. The rate we see from the knots selected by the package is significantly lower. We believe this occurs because the MLE analysis is not only attempting to find the optimal number of knots, but also select their optimal position, which allows for smaller number of knots. For instance, this is a known phenomenon in interpolation theory \cite{Phillips}.

\begin{figure}
		\centering
		\begin{subfigure}[t]{0.45\textwidth}
			\centering
			\includegraphics[width=\textwidth]{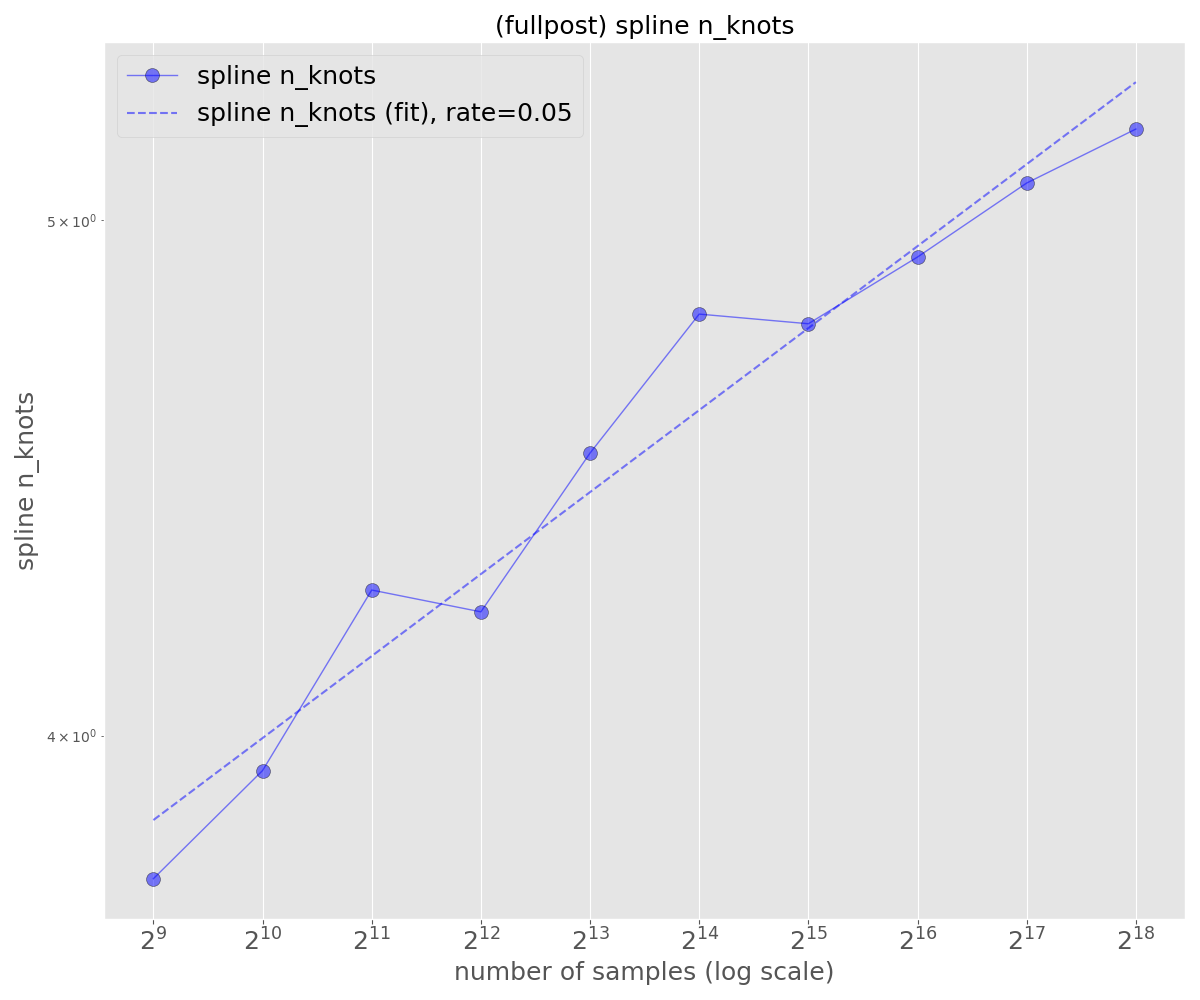}
		\end{subfigure}
		~~
		\begin{subfigure}[t]{0.45\textwidth}
			\centering
			\includegraphics[width=\textwidth]{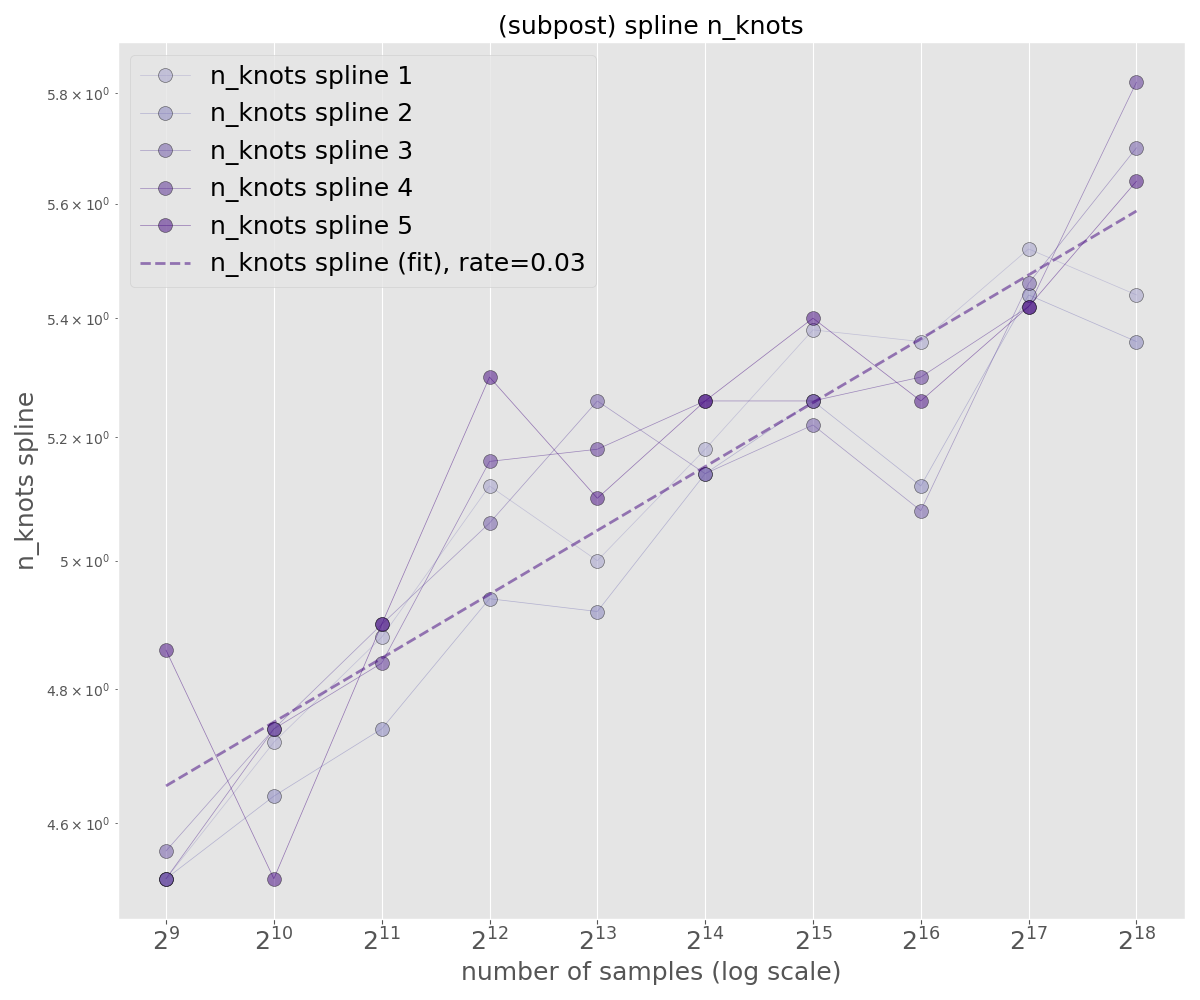}
		\end{subfigure}
		\caption{\footnotesize Expected number of spline knots.}\label{fig::flight_knots}
\end{figure}

\paragraph{Remark on the \texttt{logspline} package.} We found that working with the logspline package was challenging, due to its instability. The package allows the user to set the parameter \texttt{nknots} that plays the role of the minimum number of distinct knots used by the MLE, which then incrementally increases the number of knots from \texttt{nknots} to \texttt{max\_knots} and picks the optimal number of knots corresponding to the largest MLE. In our experiment, we picked the maximal number of knots to be \texttt{max\_knots} $\sim q+n^{\frac{1}{2q+1}}$, while the minimum number of distinct knots was always set as $3$, meaning that two were repeated. This was done on purpose because in some cases the logspline estimator would diverge for any value of the parameter \texttt{nknots}$\geq4$, which would make conducting 50 experiments difficult.

\paragraph{Remark on MISE.} Note the true MISE can be bounded by
\[
\begin{aligned}
&\EE\bigg[\int (p(\alpha|X)-\hat{p}(\alpha|X_1,\dots,X_m))^2 d\alpha\bigg]\\
& \leq \EE\bigg[\int (p(\alpha|X)-\hat{p}(\alpha|X))^2 d\alpha\bigg]
+\EE\bigg[\int (\phat(\alpha|X)-\hat{p}(\alpha|X_1,\dots,X_m))^2 d\alpha\bigg]\\
\end{aligned}
\]

We however evaluate the convergence rate of the second term on the right-hand side. Thus, though we cannot estimate the rate of convergence of the first term on the right-hand side above, as long as it converges with the same rate or faster than the second term, the empirical rate of convergence of MISE will be equal to that of \eqref{empmise}.

\section{Appendix}\label{section:appendix}
	\subsection{Notation and hypotheses}\label{section:not_hyp}
	
	For the convenience of the reader we collect in this section all hypotheses and results relevant to our analysis and present the notation that is utilized throughout the article.
	
	We assume that the probability densities $p_1,\dots,p_M$ satisfy the following hypotheses:
	\begin{enumerate}
	 	\renewcommand{\labelenumi}{\textbf{(\theenumi)}}
	 	\renewcommand{\theenumi}{H\arabic{enumi}}
		\renewcommand{\labelenumi}{\textbf{(\theenumi)}}
		\renewcommand{\theenumi}{H\arabic{enumi}}
		
		\item\label{hyp1:unifconvNh} The number of samples for each subset are parameterized by a governing parameter $n$ as follows:
		\[
		\begin{aligned}
		\N(n)&=\{N_1(n),N_2(n),N_3(n),\ldots,N_M(n)\}:\mathbb{N}\to\mathbb{N}^M\\
		\end{aligned}
		\]
		such that for all $m \in \{1,2,\ldots, M\}$
		\begin{equation*}
		D_1\leq \frac{N_m}{n}\leq D_2
		\end{equation*}
		Note that $C_1\|\N(n)\| \leq  N_{m}(n) \leq C_2\|\N(n)\| $.
		
		\item\label{hyp2:Lchoice}
		For each $m\in \{1,\dots,M\}, \ k_1=k_2=\dots=k_M=k$ for some fixed $k$ in $\NN$. The number of knots for each $m$ are parameterized by $n$ as follows:
		\begin{equation*}
		\K(n)=\{K_1(n),K_2(n),K_3(n),\ldots,K_M(n)\}:\mathbb{N}\to\mathbb{N}^M\\
		\end{equation*}
		where $K_m(n) + 1$ is the number of knots for B-splines on the interval $[a,b]$ and thus 
		\[
		\JJ(n)=\{J_1(n),J_2(n),J_3(n),\ldots,J_M(n)\}:\mathbb{N}\to\mathbb{N}^M, \quad \text{with} \quad 
		J_m(n)=K_m(n)-k+1
		\]
		and there exists $\lambda(n)$ such that $\ds \lim_{n\to \infty}\lambda(n) = \infty$ and for some $0<\epsilon<\tfrac{1}{2}$
		\[
		\quad \lim_{n\to \infty}\frac{\lambda(n)}{N_m(n)^{1/2-\epsilon}}=0, \quad \text{and} \ c_1 \leq \frac{J_m(n)}{\lambda(n)}\leq c_2.
		\]

		\item\label{hyp3:hnotation} For the knots $T_{K_m(n)}=(t_{i}^{m})_{i=0}^{K_m(n)}$, we write
		\begin{equation*}
		\bar{h}_m=\max_{k-1 \leq i \leq K_{m}(n)-k}(t_{i+1}^{m}-t_{i}^{m}) \quad \text{and} \quad \underline{h}_{m}=\min_{k-1 \leq i \leq K_{m}(n)-k}(t_{i+1}^{m}-t_{i}^{m}).
		\end{equation*}

		\item\label{hyp4:pcond1} We have either a) for each $m \in \{1,\dots,M\}$,  $p_m \in C^{k}(\RR)$ in which case we set $q=k$ or b) there exists $1 \leq r_0 < k $ such that $p_m \in C^{r}(\RR)$ while for some $m_0$, $p_{m_0} \notin C^{r_0+1}(\RR)$, in which case we set $q=r$.

		\item\label{hyp6:hmaxcond} For each subset $\textbf{x}_{m}$, the B-splines are created by choosing knots that satisfy the following minimal distance rule:
		\[
		 c_1 n^{\alpha} \leq J_m(n) \leq c_2 n^{\alpha}, \quad \alpha \in [0,\tfrac{1}{2})
		\]
		\[
		c_1 J_m^{-1}(n) \leq \underline{h}_m(n)  \leq \bar{h}_m \leq c_2 J_m^{-1}(n).
		\]

	\end{enumerate}

	Here is a brief explanation of the hypotheses. \eqref{hyp1:unifconvNh}  requires that the number of samples across subset posteriors goes to infinity at a similar rate. \eqref{hyp2:Lchoice} is a generalization of the requirement in \cite{StoneKoo,Stone89,Stone90} on the number of knots as a function of the number of samples. \eqref{hyp3:hnotation} presents the notation for the smallest and largest distance between consecutive knots. \eqref{hyp4:pcond1} is a regularity condition on the densities which serves as a compatibility condition with the choice of the order of splines used in the approximation. \eqref{hyp6:hmaxcond}, which is compatible with \eqref{hyp2:Lchoice}, requires the number of knots to grow with the number of samples at a specific rate and in addition imposes bounds on the minimal and maximal distance between knots.
	
	\subsection{B-Splines}
	In this section we will define the logspline family of densities and present an overview of how the logspline density estimator is chosen for the density $p$. The idea behind logspline density estimation of the unknown density $p$ is that the logarithm of $p$ is estimated by a spline function. Thus, we start by introducing basis splines, or B-splines, whose linear combination generates the set of splines of a given order.
	
	\begin{definition}\label{divdifdef}
		The $k$-th divided difference of a function $g:[a,b]\to \RR$ at the knots $t_{0},\dots,t_{k}$ is the leading coefficient of the interpolating polynomial $q$ of order $k+1$ that agrees with $g$ at those knots. We denote this number by
		\begin{equation*}
		[t_{0},\dots,t_{k}]g.
		\end{equation*}
	\end{definition}
	
	\noindent For example, to interpolate a function $g$ using only one knot, say at $t_{0}$, we obtain the constant polynomial $q(x)=g(t_{0})$. Since $g(t_{0})$ is the only coefficient, we have $[t_{0}]g=g(t_{0})$. Now suppose we have two knots $t_{0},t_{1}$. If $t_{0}\neq t_{1}$, then q is the secant line defined by the two points $(t_{0},g(t_{0}))$ and $(t_{1},g(t_{1}))$. The interpolating polynomial will be given by
	\begin{equation*}
	q(x)=g(t_{0})+(x-t_{0})\frac{g(t_{1})-g(t_{0})}{t_{1}-t_{0}}.
	\end{equation*}
	Thus, we write
	\begin{equation*}
	[t_{0},t_{1}]g=\frac{g(t_{1})-g(t_{0})}{t_{1}-t_{0}}=\frac{[t_{1}]g-[t_{0}]g}{t_{1}-t_{0}}.
	\end{equation*}
	If $t_{0}=t_{1}$, we can take the limit $t_{1} \to t_{0}$ above and obtain $[t_{0},t_{1}]g=g'(t_{0})$. By continuing these calculations for more knots yields the following result:
	
	\begin{lemma}\label{recuralgo}
		Given a function g and a sequence of knots $(t_{i})_{i=0}^{k}$, the $k$-th divided difference of $g$ is given by
		\begin{itemize}
			\item[(a)] $[t_{0},\dots,t_{k}]g=\ds\frac{g^{(k)}(t_{0})}{k!}$ when $t_{0}=\dots=t_{k}, g\in C^{k}$.
			\item[(b)]$[t_{0},\dots,t_{k}]g=\ds\frac{[t_{0},\dots,t_{r-1},t_{r+1},\dots,t_{k}]g-[t_{0},\dots,t_{s-1},t_{s+1},\dots,t_{k}]g}{t_{s}-t_{r}}$, where $t_{r}$ and $t_{s}$ are any two distinct knots in the sequence $(t_{i})_{i=0}^{k}$.
		\end{itemize}
	\end{lemma}
	
	Next, we show that B-splines can be expressed as scaled divided differences of a certain power function.
	
	\begin{definition}\label{bsplinedef}
		Let $t=(t_{i})_{i=0}^{N}$ be a nondecreasing sequence of knots. Let $1\leq k \leq N$. The j-th B-spline of order k, with $j\in \{0,1,\dots,N-k\}$, for the knot sequence $(t_{i})_{i=0}^{N}$ is denoted by $B_{j,k,t}$ and is defined by the rule
		\begin{equation*}
		B_{j,k,t}(x)=(t_{j+k}-t_{j})[t_{j},\dots,t_{j+k}]\max_{y}\{(y-x)^{k-1},0\}.
		\end{equation*}
	\end{definition}

	\noindent If the context is clear, we will write $B_{j}$ instead of $B_{j,k,t}$. A direct consequence of the above definition is the support of $B_{j,k,t}$.
	
	\begin{lemma}\label{Bsplinesupp}
		Let $B_{j,k,t}$ be defined as in Definition \ref{bsplinedef}. Then the support of the function is contained in the interval $[t_{j},t_{j+k})$.
	\end{lemma}
	
	\subsection{Recurrence relation and various properties}
	
	Since we stated the definition of B-splines using divided differences, we can use it to state the recurrence relation for B-splines, which is useful to prove various properties of these functions; see \cite[Ch. IX]{de Boor}.
	
	\begin{lemma}\label{Bsplinerecdef}
		Let $t=(t_{i})_{i=0}^{N}$ be a sequence of knots and let $1\leq k \leq N$. For $j\in \{0,1,\dots,N-k\}$ we can construct the $j$-th B-spline $B_{j,k}$ of order $k$ associated with the knots $t=(t_{i})_{i=0}^{N}$ as follows:
		\begin{itemize}
			\item[(1)] First we have $B_{j,1}$ be the characteristic function on the interval $[t_{j},t_{j+1})$ 
			\begin{equation*}
			B_{j,1}(x) = \left \{
			\begin{aligned}
			&1,& & x \in [t_{j},t_{j+1})&\\
			&0,& & x \notin [t_{j},t_{j+1}) &
			\end{aligned}
			\right.
			\end{equation*}
			\item[(2)] The B-splines of order $k$ for $k>1$ on $[t_{j},t_{j+k})$ are given by
			\begin{equation*}
			B_{j,k}(x)=\frac{x-t_{j}}{t_{j+k-1}-t_{j}}B_{j,k-1}(x)+\frac{t_{j+k}-x}{t_{j+k}-t_{j+1}}B_{j+1,k-1}(x).
			\end{equation*}
		\end{itemize}
	\end{lemma}
	
	The recurrence relation provides further information about B-splines. $B_{j,1}$ is a characteristic function, or otherwise piecewise constant. By Lemma \ref{Bsplinerecdef} (b), since the coefficients of $B_{j,k-1}$ are linear functions of $x$, we have $B_{j,2}$ is a piecewise linear function on $[t_{j},t_{j+2})$. Therefore, inductively we have $B_{j,3}$ is a piecewise parabolic function on $[t_{j},t_{j+3})$, $B_{j,4}$ is a piecewise polynomial of degree 3 on $[t_{j},t_{j+4})$ and so on. This implies that the set generated by B-splines is a subset of the set of piecewise polynomials with breaks at the knots $(t_{i})_{i=0}^N$. In particular, it is the set of piecewise polynomials with certain break and continuity conditions at the knots and this equality occurs on a smaller interval, called the basic interval, formally defined below.
	
	\begin{definition}\label{basicint}
		Suppose $t=(t_{0},\dots,t_{N})$ is a nondecreasing sequence of knots. Then for the B-splines of order $k$, with $2k<N+2$, that arise from these knots, we define $I_{k,t}=[t_{k-1},t_{N-k+1}]$ and call it the basic interval.
	\end{definition}
	
	\begin{remark}
		In order for this definition to be correct, we need to extend the B-splines and have them be left continuous at the right endpoint of the basic interval. For the $N-k+1$ B-splines of order $k>1$, the basic interval is defined in such a way so that at least two of them are always supported on any subinterval of $I_{k,t}$.
	\end{remark}
	
	An important property of B-splines is that they form a partition of unity as stated in the following lemma; see \cite[Ch. IX]{de Boor}.
	
	\begin{lemma}
		Let $B_{j,k,t}$ be the function as given in Definition \ref{bsplinedef} for the knot sequence $t=(t_{i})_{i=0}^{N}$. Then the following hold:
		\begin{itemize}
			\item[(a)] $B_{j,k,t}(x)>0$ for $x\in(t_{j},t_{j+k})$.
			\item[(b)] (Marsden's Identity) For any $\alpha\in \R$, we have $(x-\alpha)^{k-1}=\sum_{j}\psi_{j,k}(\alpha)B_{j,k,t}(x)$, where $\psi_{j,k}(\alpha)=(t_{j+1}-\alpha)\dots(t_{j+k-1}-\alpha)$ and $\psi_{j,1}(\alpha)=1$.
			\item[(c)] $\sum_{j}B_{j,k,t}=1$ on the basic interval $I_{k,t}$.
		\end{itemize}
	\end{lemma}
	
	\begin{remark}\label{CurryScho}
		Marsden's Identity says something very important. That all polynomials of order $k$ are contained in the set generated by the B-splines $B_{j,k}$. Another consequence of Marsden's Identity is for B-splines of order $k$ given a sequence of knots $(t_{i})_{i=0}^{N}$, we have that the number of continuity conditions at $t_{i}$ plus the multiplicity of $t_{i}$ equals $k$. Therefore, for a simple knot $t_{i}$, any B-spline of order $k$ is continuous and has $k-2$ continuous derivatives. On the other hand, if $t_{i}$ has multiplicity $k$, any $k$-th order B-spline has a discontinuity there.
	\end{remark}
	
	\begin{definition}\label{Bsplinespace}
		Let $(t_{i})_{i=0}^{N}$ be a sequence of knots such that $t_{0}=\dots=t_{k-1}$ and $t_{N-k+1}=\dots=t_{N}$, where $1\leq k\leq N$. Let $B_{j,k,t}$ be the B-splines as defined in \ref{bsplinedef} with knot sequence $t=(t_{i})_{i=0}^{N}$. The set generated by the sequence $\{B_{j,k,t}:\text{all} \ j\}$, denoted by $\SSS_{k,t}$, is the set of splines of order $k$ with knot sequence $t$. In symbols we have
		\begin{equation*}
		\SSS_{k,t}=\left\{ \sum_{j}a_{j}B_{j,k,t}:a_{j}\in \R, \ \text{all} \ j  \right\}
		\end{equation*}
	\end{definition}
	
	\begin{remark}\label{Bsplinedense}
		Fix an interval $[a,b]$. Let $T_{N}=(t_{i})_{i=0}^{N}$ be a sequence as in definition \ref{Bsplinespace} with $t_{0}=a$ and $t_{N}=b$, where $N\in \NN$. The choice in definition \ref{Bsplinespace} implies that
		\begin{equation*}
		\bigcup_{N\in \NN}S_{k,T_{N}} \quad \text{is dense in} \quad C([a,b])
		\end{equation*}
	\end{remark}
	
	\subsection{Derivatives of B-spline functions}
	
	The derivative of a $k$-th order B-spline is directly associated with B-splines of order $k-1$. To see this we use the recurrence relation which leads us to the following theorem:
	
	\begin{theorem}\label{Bsplinederiv}
		Let $B_{j,k,t}$ be the function as defined in \ref{bsplinedef}. The support of $B_{j,k,t}$ is the interval $[t_j,t_{j+k})$. Then the following equation holds on the open interval $(t_j,t_{j+k})$
		\begin{equation*}
		\frac{d}{d \theta}B_{j,k,t}(\theta)=
		\begin{cases}
		0, \ &k=1\\
		(k-1)\left( \dfrac{B_{j,k-1,t}(\theta)}{t_{j+k-1}-t_j}-\dfrac{B_{j+1,k-1,t}(\theta)}{t_{j+k}-t_{j+1}} \right), \ &k>1
		\end{cases}
		\end{equation*} 
	\end{theorem}
	
	\begin{proof}
		The proof is done by induction on $k$. For $k=1$ it is straightforward since $B_{j,1,t}$ is a constant on $(t_{j},t_{j+1})$ and for $k>1$ we use the recurrence relation described in lemma \ref{Bsplinerecdef}.
	\end{proof}
	
	Using the above formula we can obtain bounds for higher derivatives of B-splines. By construction of the space $\SSS_{k,t}$, the B-splines form a partition of unity on $[t_{0},t_{N}]$, and since they are strictly positive on the interior of their supports, we have that each B-spline is bounded by 1 for all $\theta$. Furthermore, by induction we can prove the following lemma:
	\begin{lemma}\label{Bsplinederbound}
		Let $t=(t_{i})_{i=0}^{N}$ be a sequence of knots as in definition \ref{Bsplinespace} and $B_{j,k,t}$ be the function as defined in \ref{bsplinedef}. Let $h_{N}=\ds\min_{k\leq i\leq N-k+1 }(t_{i}-t_{i-1})$ and $\alpha$ be a positive integer such that $\alpha<k-1$. Then, on the open interval $(t_{j},t_{j+k})$ we have
		\[
		\sup_{\theta \in (t_{j},t_{j+k})}\left|\frac{d^{\alpha}}{d\theta^{\alpha}}B_{j,k,t}(\theta)\right|\leq \dfrac{2^{\alpha}}{h_{N}^{\alpha}}\dfrac{(k-1)!}{(k-\alpha-1)!}, \text{for any } j
		\]
	\end{lemma}
	
	\begin{proof}
		The result is obtained by fixing $k$ and doing induction on $\alpha$.
	\end{proof}
	
	\subsection{Lagrange interpolation}
	
	The following two theorems can be found in \cite{Atkinson}.
	\begin{theorem}\label{approx}
		The Lagrange interpolating polynomial $q(x)$ of $f:[a,b]\rightarrow \mathbb{R}$, given distinct points 
		$a=x_{0}<x_{1}<...<x_{l}=b$  and $l+1$ ordinates $y_{i}=f(x_{i})$, $i=0,\dots,l$, is defined as
		\begin{equation}\label{lagrp}
		q(x)=\sum_{i=0}^{l}y_{i}l_{i}(x) \quad \text{with} \quad l_{i}(x) =\ds\prod_{j\neq
			i}\left(\frac{x-x_{j}}{x_{i}-x_{j}}\right)\,,\;i=0,1,\dots,l.
		\end{equation}
	\end{theorem}

	\begin{theorem}\label{interpthm}
		Let $f,q$ and $(x_i,y_i)$, $i=0,\dots,l$, be as in Theorem \ref{approx}. Suppose that $f$ has $l+1$ continuous derivatives on $(a,b)$. Then for every $x \in [a,b]$ there exists $\xi \in (a,b)$ such that
		\begin{equation*}
		f(x)-q(x)= \dfrac{\prod_{i=0}^{l}(x-x_{i})}{(l+1)!}f^{(l+1)}(\xi).
		\end{equation*}
	\end{theorem}
	
	We next state an elementary lemma that provides the estimate of the interpolation error when
	information on the derivatives of $f$ is available.

	\begin{lemma}\label{estinterrlmm}
		Let $f, q$, and $(x_i,y_i)$, $i\in \{0,\dots,l\}$, be as in Theorem \ref{approx}. Suppose that	for some $C>0$ we have $\ds \sup_{x \in [a,b]}|f^{(l+1)}(x)| \leq C$, and $\overline{\Delta x}=\sup_{i}(x_{i+1}-x_{i})$. Then
		\begin{equation*}
		\max_{x \in [a,b]}|f(x)-q(x)|\leq C  \dfrac{(\overline{\Delta x})^{l+1}}{4(l+1)}.
		\end{equation*}
	\end{lemma}

\end{document}